\documentclass[reqno,11pt]{amsart}
\usepackage[utf8x]{inputenc}
\usepackage[unicode]{hyperref}
\usepackage{mathrsfs,bbm}
\usepackage{graphicx}

\usepackage[english]{babel}
\usepackage{verbatim}
\usepackage{supertabular}
\usepackage{stmaryrd}
\usepackage{longtable}
\usepackage{varioref}
\usepackage{calc}
\usepackage{ifthen}
\usepackage{indentfirst}
\usepackage{fancyhdr}
\usepackage{enumerate}
\usepackage{float}
\usepackage{xcolor}
\usepackage{multicol}
\usepackage[BCOR4mm,DIV10]{typearea}
\usepackage[refpage]{nomencl/nomencl}
\usepackage{amssymb}
\usepackage{amsxtra}
\usepackage{amsmath} 
\usepackage{amsfonts}
\usepackage{etoolbox}
\usepackage{xfrac}
\usepackage{nicefrac}

\usepackage[normalem]{ulem}

\usepackage{todonotes}
\usepackage{ucs}
\usepackage[T1]{fontenc}
\usepackage{here}
\usepackage{tikz}
\usepackage{subcaption}
\usepackage{mathtools}
\usepackage{subcaption}
\usetikzlibrary{quotes,angles}

\newcommand{\Sp}{\mathbb{S}}
\newcommand{\VMO}{\text{VMO}}

\newcommand{\Sl}{\R / L\Z}
\newcommand{\Sc}{\R / \Z}
\newcommand{\diffsn}[1]{\left|s_{n,#1+1}-s_{n,#1}\right|}
\newcommand{\sn}[1]{s_{n,#1}}
\newcommand{\gsn}[1]{\gamma\left(s_{n,#1}\right)}
\newcommand{\gstrichsn}[1]{\gamma'\left(s_{n,#1}\right)}
\newcommand{\distg}[2]{\dist\left(l\left(#2\right),\gamma\left(#1\right)\right)}
\newcommand{\diffgamma}[2]{\left|\gamma\left(#1\right)-\gamma\left(#2\right)\right|}
\newcommand{\E}{\mathcal{E}}
\newcommand{\beq}[1][ ]{\stackrel{\mathmakebox[\widthof{(1.1)}]{#1}}=}
\newcommand{\bleq}[1][ ]{\stackrel{\mathmakebox[\widthof{(1.1)}]{#1}}\leq}
\newcommand{\bgeq}[1][ ]{\stackrel{\mathmakebox[\widthof{(1.1)}]{#1}}\geq}
\newcommand{\ble}[1][ ]{\stackrel{\mathmakebox[\widthof{(1.1)}]{#1}}<}

\newcommand{\wg}{\omega_{\gamma'}}
\newcommand{\rope}{\mathcal{R}}

\numberwithin{equation}{section}
\allowdisplaybreaks

\newtheorem{theorem}{Theorem}[section]
\newtheorem{proposition}[theorem]{Proposition}
\newtheorem{lemma}[theorem]{Lemma}
\newtheorem{corollary}[theorem]{Corollary}
\newtheorem{definition}[theorem]{Definition}
\newtheorem{remark}[theorem]{Remark}

\newtheorem*{theorem*}{Theorem}
\hypersetup{
breaklinks=true,
colorlinks=true,
linkcolor=blue,
citecolor=blue,
urlcolor=blue,
}

\DeclareMathOperator{\supp}{supp}

\DeclareMathOperator{\dist}{dist}

\newcommand{\B}{\mathcal{B}}

\newcommand{\RL}{\mathcal{R}}

\newcommand{\BL}{\mathcal{B}}

\newcommand{\EL}{\mathcal{E}}
\newcommand{\KL}{\mathcal{K}}

\newcommand\Id{{{\rm Id}}}

 \renewcommand\S{{\mathbb S}}

\newcommand{\omitted}[1]{}

\newcommand{\Fo}{\,\,\,\text{for }\,\,}
\newcommand{\Foa}{\,\,\,\text{for all }\,\,}

\newcommand{\AND}{\,\,\,\text{and }\,\,}

\newcommand{\heikodetail}[1]{}

\usepackage[left,mathlines,pagewise]{lineno}
\newcommand*\patchAmsMathEnvironmentForLineno[1]{%
\expandafter\let\csname old#1\expandafter\endcsname\csname #1\endcsname
\expandafter\let\csname oldend#1\expandafter\endcsname\csname end#1\endcsname
\renewenvironment{#1}%
{\linenomath\csname old#1\endcsname}%
{\csname oldend#1\endcsname\endlinenomath}}%
\newcommand*\patchBothAmsMathEnvironmentsForLineno[1]{%
\patchAmsMathEnvironmentForLineno{#1}%
\patchAmsMathEnvironmentForLineno{#1*}}%
\AtBeginDocument{%
\patchBothAmsMathEnvironmentsForLineno{equation}%
\patchBothAmsMathEnvironmentsForLineno{align}%
\patchBothAmsMathEnvironmentsForLineno{flalign}%
\patchBothAmsMathEnvironmentsForLineno{alignat}%
\patchBothAmsMathEnvironmentsForLineno{gather}%
\patchBothAmsMathEnvironmentsForLineno{multline}%
}

\renewcommand{\d}{\ensuremath{\,\mathrm{d}}}

\newcommand{\eps}{\ensuremath{\varepsilon}}

\newcommand{\F}[1][\eps]{F}

\newcommand{\g}{\ensuremath{\gamma}}

\newcommand{\N}{\ensuremath{\mathbb{N}}}

\newcommand{\R}{\ensuremath{\mathbb{R}}}
\renewcommand{\rho}{\ensuremath{\varrho}}

\newcommand{\TP}{\mathrm{TP}}

\newcommand{\Z}{\ensuremath{\mathbb{Z}}}

\graphicspath{{./mathematica/}}

\title[Gamma-limit of discrete tangent-point energies]{Tangent-point 
energies and ropelength as Gamma-limit of discrete 
tangent-point energies on biarc curves}
\author{Anna Lagemann}
\address[A.~Lagemann]{
\newline%
RWTH Aachen University,\newline%
Institut f\"ur Mathematik,\newline%
Templergraben 55,
52062 Aachen,
Germany}
\email{lagemann@eddy.rwth-aachen.de}%
\author{Heiko von der Mosel}
\address[H.~von~der~Mosel]{
\newline%
RWTH Aachen University,\newline%
Institut f\"ur Mathematik,\newline%
Templergraben 55,
52062 Aachen,
Germany}
\email{heiko@instmath.rwth-aachen.de}

\date{\today}

\DeclareRobustCommand{\SkipTocEntry}[5]{}
\begin{document}


\begin{abstract}
Using interpolation with biarc curves we prove $\Gamma$-convergence of discretized tangent-point energies to the continuous tangent-point energies in the
  $C^1$-topology, as well as
  to the ropelength functional. As a consequence discrete almost minimizing
  biarc curves converge to ropelength minimizers, and to minimizers of the continuous
  tangent-point energies. In addition, taking point-tangent data
  from a given $C^{1,1}$-curve $\g$, 
  we establish convergence of the discrete
  energies evaluated on biarc curves interpolating these data,
  to the continuous  
  tangent-point energy of $\g$, together with an explicit convergence
  rate.
\end{abstract}

\maketitle


\section{Introduction}\label{sec:intro}
The ropelength\footnote{This name is coined after the
mathematical question, how long a thick rope has to be in order
to tie it into a knot.}
of a closed arclength parametrized
curve $\g:\R/L\Z\to\R^3$ is defined as the 
quotient of its length and thickness,
\begin{equation}\label{eq:ropelength}
\RL(\g):=\frac{\mathscr{L}(\g)}{\triangle[\g]}=\frac{L}{\triangle[\g]}.
\end{equation}
Here,  for variational considerations
the thickness $\triangle[\g]$ is  most conveniently expressed
following Gonzalez and Maddocks \cite{gonzalez-maddocks_1999} -- without
any regularity assumptions on the curve $\g$ -- as
\begin{equation}\label{eq:thickness}
\triangle[\g]:=\inf_{s\not=t\not=\tau\not=s}R(\g(s),\g(t),\g(\tau)),
\end{equation}
where $R(x,y,z)$ denotes the circumcircle radius of the three 
points $x,y,z\in\R^3$. 
Motivated by numerous applications in the Natural Sciences 
ropelength is used in numerical computations (see
\cite{carlen-etal_2005,ashton-etal_2011,cantarella-etal_2012b,klotz-maldonado_2021} and the references therein)
to mathematically model long and
slender objects such as strings or macromolecules that
do not self-intersect. 
In fact, it was proved rigorously in 
\cite{gonzalez-etal_2002,cantarella-etal_2002}  that a curve of 
finite ropelength is embedded and of class $C^{1,1}(\R/L\Z,\R^3)$,
which means that its curvature exists and is bounded a.e. on $\R/L\Z$. 
Moreover, a curve $\g$ with positive thickness $\triangle[\g]>0$
is surrounded by an embedded
tube with radius equal to $\triangle[\g]$ as shown in 
\cite[Lemma 3]{gonzalez-etal_2002}, which justifies the use of
the non-smooth quantity
$\triangle[\cdot]$ as steric excluded volume constraint.

The minimization 
over all triples of curve points 
to evaluate thickness in \eqref{eq:thickness} 
is costly, which lead to the idea
to replace minimization by integration; see 
\cite[p. 4773]{gonzalez-maddocks_1999}. One such integral
energy is the \emph{tangent-point energy}
\begin{equation}\label{eq:tan-point}
\TP_q(\g)
:=\iint_{(\R/L\Z)^2}\frac1{r^q_\textnormal{tp}(\g(s),\g(t))}\d s \d t,\quad
q\ge 2,
\end{equation}
where the circumcircle radius is now replaced by the
\emph{tangent-point radius} 
\begin{equation}\label{eq:tp-radius}
r_\textnormal{tp}(\g(s),\g(t))=
\frac{|\g(s)-\g(t)|^2}{2\dist(\g(s)+\R\g'(s),\g(t))},
\end{equation} 
i.e.,
the radius of the unique circle through the points $\g(s)$ and $\g(t)$
that is in addition tangent to the curve $\g$ at $\g(s)$.
Also this energy implies 
self-avoidance and has regularizing properties. It was
shown in \cite{strzelecki-vdm_2012} that if $\TP_q(\g)$ 
is finite for some
$q>2$, then
$\g$ is embedded and of class $C^{1,1-\frac2{q}}(\R/L\Z,\R^3)$. Later,
Blatt \cite{blatt_2013b} improved this regularity to the optimal 
fractional Sobolev\footnote{For the definition 
see Appendix \ref{proof_convolution_convergence}; a condensed
selection of
pertinent results regarding periodic fractional Sobolev
spaces can be found, e.g.,  in \cite[Appendix A]{knappmann-etal_2022}.}
regularity $W^{2-\frac1{q},q}(\R/L\Z,\R^3)$,
 which
actually characterizes curves of finite $\TP_q$-energy.
The knowledge of the exact energy space was then used to establish
continuous differentiability of the tangent-point energy 
\cite[Remark
3.1]{blatt-reiter_2015a}, \cite{wings_2018}, 
and to find $\TP_q$-critical knots by means of Palais's
symmetric criticality principle \cite{gilsbach_2018}. Very recently,
long-time existence for a suitably regularized gradient flow for
$\TP_q$ was shown via a minimizing movement scheme \cite{matt-etal_2022}.

But the tangent-point energy was also used in numerical simulations.
Bartels et al added a desingularized
variant of the $\TP_q$-energy in 
\cite{bartels-etal_2018,bartels-reiter_2021} as a self-avoidance term
to  the bending energy to find elastic knots. The impressive simulations
of Crane et al. in \cite{yu-etal_2021b} use the $\TP_q$-energy as well
to avoid self-intersections, a higher dimensional tangent-point
energy allows for computations on self-avoiding surfaces; see
\cite{yu-etal_2021a}.

In the present paper we address the mathematical question of
variational convergence of suitably discretized tangent-point energies
towards the continuous $\TP_q$-energy, as well as towards
ropelength. To account for the tangential information encoded in
the tangent-point radius in \eqref{eq:tan-point} 
on the discrete level we use \emph{biarcs}, i.e., pairs of circular arcs
as in \cite{smutny_2004,gonzalez-etal_2002b,carlen-etal_2005}, 
which on the one hand,
can interpolate
\emph{point-tangent data}
\begin{equation}\label{eq:point-tangent-data}
\big(\,\g(s_i),\g'(s_i)\,\big)\,\in\R^3\times\S^2\quad\Fo
i=1,\ldots,n
\end{equation}
of a given arclength parametrized
curve $\g\in C^1(\R/L\Z,\R^3)$.
Every \emph{biarc curve}
$\beta$ consisting of $n$ consecutive biarcs is therefore a
$C^{1,1}$-interpolant of the curve $\g$. On the other hand,
every biarc curve produces point-tangent data
\begin{equation}\label{eq:biarc-point-tangent-data}
(q_i,t_i)\in\R^3\times\S^2\quad\Fo i=1,\ldots,n,
\end{equation}
on its own, namely the
points $q_i$ and unit-tangents $t_i$ at every
junction of  two consecutive biarcs.
In order to avoid
degeneracies  we restrict
to those biarc curves $\beta$ whose biarcs have lengths $\lambda_i$
that are controlled in terms of the curve's length $\mathscr{L}(\g)$
by means of the inequality
\begin{equation}\label{eq:biarc-lengths}
\frac{\mathscr{L}(\g)}{2n}\le\lambda_i\le\frac{2\mathscr{L}(\g)}{n}\quad
\Fo i=0,\ldots,n-1.
\end{equation}
Let $\BL_n$ be the class of biarc curves $\beta$ satisfying
\eqref{eq:biarc-lengths}.
Accordingly,
we define in a parameter-invariant fashion
the \emph{discrete tangent-point energy} $\EL_q^n$ 
for $n\in\N$ and $q\in [2,\infty)$
on closed $C^1$-curves $\g$ as
\begin{align}\label{eq:discrete-tan-point}
\E_q^n(\gamma):=\begin{cases}
\displaystyle \sum_{i=0}^{n-1} \sum_{j=0,j\neq i}^{n-1} 
\left(\frac{2\;\dist(l(q_j),q_i)}{|q_i-q_j|^2}\right)^q \lambda_i \lambda_j, & \text{if }\gamma \in \mathcal{B}_n\\
\infty, & \, \text{otherwise}
\end{cases}
\end{align}
with  the straight lines
$l(q_i):=q_i + \R t_i $ for $i=0, ..., n-1$.
Notice that both $\TP_q$ and $\E_q^n$ are invariant
under reparametrization of the curves, and
they have the same scaling behaviour,
\begin{equation}\label{scalings}
\TP_q(d\g)=d^{2-q}\TP_q(\g)\AND 
\E_q^n(d\g)=d^{2-q}\E_q^n(\g)\quad\Foa d>0.
\end{equation}
We restrict to \underline{i}njective $C^1$-curves 
that are parametrized by 
\underline{a}rclength, denoted as the subset $C^1_\textnormal{ia}$ to
state our main results. 
\begin{theorem}[$\Gamma$-convergence to tangent-point energy]
\label{gamma_convergence}
For $q>2$ and $L>0$
the discrete tangent-point energies $\EL^n_q$ $\Gamma$-converge
to the tangent-point energy $\TP_q$ on the space
$C^1_\textnormal{ia}(\R/L\Z,\R^3)$
with respect to the $\|\cdot\|_{C^1}$-norm as $n\to\infty$, i.e., 
\begin{align}\label{eq:gamma-convergence-tp}
\E_q^n \stackrel{\Gamma}{\underset{n \to \infty}{\longrightarrow}} 
\TP_q \quad \text{on } 
\left(C_\textnormal{ia}^1\left(\R/L\Z, \R^3 \right), \|\cdot \|_{C^1}\right).
\end{align}
\end{theorem}
As an immediate consequence we infer the convergence
of almost minimizers in a given knot class $\KL$
of the discrete energies $\EL^n_q$ to 
a minimizer of the continuous tangent-point energy $\TP_q$
in the same knot class $\KL$.
\begin{corollary}[Convergence of discrete almost minimizers]
\label{cor:convergence-mini-tp}
Let $q>2$, $L>0$,  and $\mathcal{K}$ be a tame knot class and 
$b_n \in \mathcal{C}^*:=C_\textnormal{ia}^1\left(\R/L\Z, \R^3 \right)\cap\mathcal{K}
$ with
\begin{align*}
\big|\inf_{\mathcal{C}^*} \E_q^n - \E_q^n\left(b_n\right) \big| 
\to 0 \quad \text{and} \quad \| b_n - 
\gamma\|_{C^1} \to 0 \,\,\text{ as } n \to \infty.
\end{align*}
Then $\gamma$ is a minimizer of $\TP_q$ in $\mathcal{C}^*$
and $\lim_{n \to \infty} \E_q^n(b_n)=
\TP_q(\gamma)$.
Furthermore, it holds that $\gamma \in W^{2-\frac{1}{q},q}
(\R/L\Z,\R^3)$.
\end{corollary}
Moreover, the discrete tangent-point energies can also be used to
approximate the non-smooth ropelength functional $\rope$ in the 
sense of $\Gamma$-convergence. 
\begin{theorem}[$\Gamma$-convergence to ropelength]
\label{gamma_conv_rope}
It holds that  $L^{\frac{n-2}{n}}\left(\E_n^n\right)^{\frac{1}{n}}  
\stackrel{\Gamma}{\underset{n \to \infty}{\longrightarrow}} \rope$ on 
$(C_\textnormal{ia}^{1,1}(\Sl,\R^3),\|\cdot\|_{C^1} )$.
\end{theorem}
Also here we can state the convergence of almost minimizers to
ropelength minimizing curves in a prescribed knot class, which could
be of computational relevance for the minimization of ropelength.
\begin{corollary}[Discrete almost minimizers approximate ropelength minimizers]
\label{conv_mini_ropelength}
Let $\mathcal{K}$ be a tame knot class and 
$b_n \in \mathcal{C}^{**}:=
C_\textnormal{ia}^{1,1}\left(\Sl, \R^3 \right)\cap\mathcal{K}
$ with
\begin{align*}
\big|\inf_{ \mathcal{C}^{**}} 
\E_n^n - \E_n^n\left(b_n\right) \big| \to 0 \quad \text{and} \quad 
\| b_n - \gamma\|_{C^1} \to 0 \text{ as } n \to \infty.
\end{align*}
Then $\gamma$ is a minimizer of $\rope$ in 
$\mathcal{C}^{**}$ and 
$\lim_{n \to \infty} \E_n^n(b_n)=\rope(\gamma)$.
\end{corollary}
To the best of our knowledge, the only known contributions on 
variational convergence of discrete energies  to continuous
knot energies are the $\Gamma$-convergence
results of Scholtes.
In \cite{scholtes_2014b} he proves $\Gamma$-convergence of a discrete
polygonal
variant of the M\"obius energy to the classic M\"obius energy introduced by
O'Hara \cite{ohara_1991a}.
This result was strengthened later by
Blatt \cite{blatt_2019b}. In \cite{scholtes_2014a,scholtes_2018}
he proved the  $\Gamma$-convergence of polygonal versions of 
ropelength and of integral Menger curvature to ropelength and
to continuous integral Menger curvature, respectively.
It remains open at this point if stronger types of variational
convergence such as Hausdorff convergence of sets
of almost minimizers can be shown
for the non-local knot energies treated here, as was, e.g.,
established  in \cite{scholtes-etal_2022}
for the classic bending energy under clamped
boundary conditions.
It would be also interesting to set up
a numerical scheme for the discretized tangent-point energies 
$\E^n_q$ to
numerically approximate ropelength minimizers, in comparison to the 
simulated annealing computations in 
\cite{smutny_2004, carlen-etal_2005},
or to compute discrete (almost) minimizers of the tangent-point energy. 
The almost linear
energy convergence rate 
established in Theorem \ref{konv_ordnung_energien} in Section
\ref{approx_energies}
is identical with the one in \cite[Proposition 3.1]{scholtes_2014b}
for Scholtes' polygonal M\"obius energy, which exceeds
the $n^{-\frac14}$-convergence rate for the 
minimal distance approximation of the M\"obius energy by
 Rawdon and Simon \cite[Theorem 1]{rawdon-simon_2006}.

The present
paper is structured as follows. In Section \ref{biarcs} 
we provide the 
necessary background on biarcs --- mainly following Smutny's work \cite{smutny_2004}.
Section \ref{approx_energies} 
is devoted to the convergence of the discretized energies $\E^n_q$
including explicit convergence rates; see Theorem 
\ref{konv_ordnung_energien}.
In Section \ref{gamma_section} we treat $\Gamma$-convergence towards the continuous
tangent-point energies, as well as convergence of discrete almost 
minimizers, 
to prove 
Theorem \ref{gamma_convergence}
and Corollary \ref{cor:convergence-mini-tp}. 
Finally, in Section \ref{sec:ropelength} 
we prove $\Gamma$-convergence to the ropelength functional, 
Theorem \ref{gamma_conv_rope} and 
convergence of discrete almost minimizers to ropelength minimizers, 
Corollary
\ref{conv_mini_ropelength}. 
In Appendix \ref{proof_convolution_convergence} 
we establish convergence of
rescaled and reparametrized convolutions in fractional Sobolev spaces. 
Appendix \ref{app:B} contains some quantitative analysis
of general $C^1$-curves, and in Appendix \ref{app:C} we prove some 
auxiliary general
results in the context of $\Gamma$-convergence.

\section{Biarcs and Biarc curves}
\label{biarcs}
The discrete tangent-point energy defined in \eqref{eq:discrete-tan-point}
of the introduction is defined on biarc curves, which are space curves 
assembled from biarcs, i.e., from pairs of circular arcs. In this section we first present the basic definitions and a general existence result
due to Smutny \cite[Chapter 4]{smutny_2004}, before specializing to balanced proper biarc interpolations needed in our convergence proofs later on. 

\begin{definition}[Point-tangent pairs and biarcs]
\label{def:tp-pairs-biarcs}
Let $\mathcal{T} :=\R^3 \times \Sp^2$ be the set of \emph{point-tangent data} $[q,t]$,
where $\Sp^2$ is the unit sphere in $\R^3$. 

{\rm (i)}\,
A \emph{point-tangent pair} is a pair of tuples of the form 
$\left(\left[q_0,t_0\right],\left[q_1,t_1\right]\right) \in \mathcal{T} \times \mathcal{T}$
with $q_0 \neq q_1$.

{\rm (ii)}\,
A \emph{biarc} $(a,\bar{a})$ is a pair of circular arcs in $\R^3$ that 
are continuously joined with continuous tangents and that
interpolate a point-tangent pair 
$\left(\left[q_0,t_0\right],\left[q_1,t_1\right]\right) 
\in \mathcal{T} \times \mathcal{T}$. The common end point $m$ of the two 
circular arcs $a$ and $\bar{a}$ is called \emph{matching point}. 
The interpolation is meant with orientation, 
such that $t_0$ points to the interior of the arc $a$ and 
$-t_1$ points to the interior of the arc $\bar{a}$; see Figure 
\ref{example_biarcs}.
\end{definition}

\begin{figure}[h]
	\centering
	\begin{minipage}[b]{0.25\textwidth}
		\includegraphics[width=\textwidth]{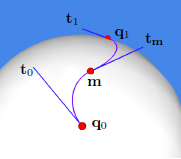}
	\end{minipage}
	\quad
	\quad
	\begin{minipage}[b]{0.25\textwidth}
		\includegraphics[width=\textwidth]{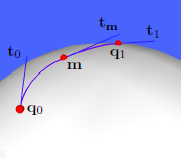}
	\end{minipage}
	\quad
	\quad
	\begin{minipage}[b]{0.25\textwidth}
		\includegraphics[width=\textwidth]{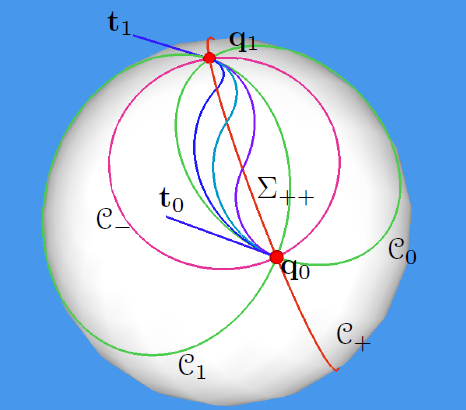}
	\end{minipage}
	\caption{\label{example_biarcs} 
	{\bf Left and middle:}  Examples of biarcs, $t_m$ is the 
	tangent at the common matching point $m$. {\bf Right:} 
	The circles $C_0,$ $C_1,$ $C_+$, and $C_-$ of Definition
\ref{def:C1C+C-}. Images taken from \cite[Figures 4.1 \& 4.3]{smutny_2004}
	by courtesy of Jana Smutny.}
\end{figure}

For two points $q_0, \; q_1 \in \R^3$ we set
$
d:=q_1 - q_0 $ and $e:=\frac{q_1-q_0}{|q_1-q_0|}=\frac{d}{|d|},$ and 
define
\begin{align}
\label{rot_matrix}
R(e):= 2 \; e \otimes e - \text{Id}=2e e^T-\Id,
\end{align}
which is a symmetric, proper rotation matrix representing
the reflection of $t$ at the unit vector $e$. Moreover,
for a point-tangent pair 
$\left(\left[q_0,t_0\right],\left[q_1,t_1\right]\right) 
\in \mathcal{T} \times \mathcal{T}$ we set
\begin{align}\label{eq:t0*t1*}
t_0^*:=R(e)t_0 \quad \text{and} \quad t_1^*:=R(e)t_1.
\end{align}

\begin{definition}\label{def:C1C+C-}
Let $\left(\left[q_0,t_0\right],
\left[q_1,t_1\right]\right) \in \mathcal{T} \times \mathcal{T}$
be a point-tangent pair.

{\rm (i)}\,
Let $\mathcal{C}_0$ be the circle through $q_0$ and $q_1$ with tangent 
$t_0$ at $q_0$ and let $\mathcal{C}_1$ be the circle through both 
points with tangent $t_1$ at $q_1$. If $t_0 + t_1^* \neq 0$, we 
denote the circle through $q_0$ and $q_1$ with tangent $t_0 + t_1^*$ 
at $q_0$ by $\mathcal{C}_+$, if $t_0 - t_1^* \neq 0$, we denote the 
circle through both points with tangent $t_0 - t_1^*$ at $q_0$ 
by $\mathcal{C}_{-}$; see Figure \ref{example_biarcs} on the right.

{\rm (ii)}\,
A point-tangent pair $\left(\left[q_0,t_0\right],
\left[q_1,t_1\right]\right) \in \mathcal{T} \times \mathcal{T}$ is 
called \emph{cocircular}, if $\mathcal{C}_0=\mathcal{C}_1$ 
as point sets. A cocircular point-tangent pair is classified as 
\emph{compatible}, if the orientations of the two circles induced by the tangents agree, and \emph{incompatible} otherwise.
\end{definition}

\begin{remark}
For a point-tangent pair $\left(\left[q_0,t_0\right],\left[q_1,t_1\right]\right) \in \mathcal{T} \times \mathcal{T}$, the compatible cocircular case is equivalent to $t_0 - t_1^* =0$. In this case, the circle $\mathcal{C}_{-}$ is not defined. The incompatible cocircular case is equivalent to $t_0 + t_1^*=0$, thus the circle $\mathcal{C}_+$ is not defined. 
\end{remark}

The following central existence result 
of Smutny not only states that interpolating
biarcs always exist, but it also characterizes geometrically
the possible locations
of the corresponding matching points depending on the type
of the point-tangent pair. For the precise statement we denote
for an arbitrary circle 
$\mathcal{C}$ through $q_0$ and $q_1$ the punctured set 
$\mathcal{C}':=\mathcal{C} \setminus \{q_0,q_1\}$.

\begin{proposition}
\label{existence_biarcs}
\cite[Proposition 4.7]{smutny_2004} 
For a given point-tangent pair \linebreak $\left(\left[q_0,t_0\right],\left[q_1,t_1\right]\right) \in \mathcal{T} \times \mathcal{T}$, we denote by $\Sigma_+ \subset \R^3$ the set of matching points of all possible biarcs interpolating the point-tangent pair.
Then:
\begin{enumerate}
\item If $\left(\left[q_0,t_0\right],\left[q_1,t_1\right]\right)$ is not cocircular, then $\Sigma_+= \mathcal{C}_{+}^{'}$.
\item If $\left(\left[q_0,t_0\right],\left[q_1,t_1\right]\right)$ is cocircular, we distinguish between two cases:
	\begin{enumerate}
	\item[\rm (a)] 
	If the point-tangent pair is compatible, then 
$\Sigma_+= \mathcal{C}_{+}^{'}=\mathcal{C}_0^{'}=\mathcal{C}_1^{'}$.
	\item[\rm (b)] If the point-tangent pair is incompatible, then $\Sigma_+$ is the sphere passing through $q_0$ and $q_1$ perpendicular to the circle $\mathcal{C}_{-}$ without the points $q_0$ and $q_1$.
	\end{enumerate}
\item $\Sigma_+$ is a straight line passing through $q_0$ and $q_1$ without the two points if and only if $t_0=t_1$ and $\left\langle t_0, e \right\rangle \neq 0$.
\item $\Sigma_+$ is a plane through $q_0$ and $q_1$ without the two points if and only if $t_0=t_1$ and $\left\langle t_0, e \right\rangle =0$.
\end{enumerate}
\end{proposition}
A particularly powerful interpolation is possible if the
location of the matching point $m\in\Sigma_+$
of the biarc is roughly ``in between'' the points $q_0$ and $q_1$. 
The following definition states this precisely for the relevant
cases (i), (ii)(a), and (iii) of Proposition~\ref{existence_biarcs}.

\begin{definition}[Desired matching point
location $\Sigma_{++}$ and proper biarcs]\label{sigma++}
{\rm (i)}\,
Let  $\left(\left[q_0,t_0\right],\left[q_1,t_1\right]\right) \in \mathcal{T} \times \mathcal{T}$ be a point-tangent pair that is not incompatible cocircular. Then we denote by $\Sigma_{++} \subset \Sigma_+$ the subarc of $\Sigma_+$ from $q_0$ to $q_1$ with the orientation induced by the tangent $t_0 + t_1^*$ (see Figure \ref{sigma_plus_plus}).

{\rm (ii)}\,
A point-tangent pair $\left(\left[q_0,t_0\right],\left[q_1,t_1\right]\right) \in \mathcal{T} \times \mathcal{T}$ is called \emph{proper} if
$
\left\langle q_1 - q_0, t_0 \right\rangle >0$ and $\left\langle q_1 - q_0, t_1 \right\rangle >0.
$

{\rm (iii)}\,
A biarc is called \emph{proper} if it interpolates a proper point-tangent pair with a matching point $m \in \Sigma_{++}$.

{\rm (iv)}\,
Let $\gamma \in C^1_{\textnormal{ia}}\left(\Sl,\R^3\right)$.
We call a biarc \emph{$\gamma$-interpolating} and \emph{balanced} if it interpolates a point-tangent pair
$
\left(\left[\gamma(s),\gamma'(s)\right],\left[\gamma(s+h),\gamma'(s+h)\right]\right),
$
such that the matching point $m_h \in \Sigma_{++}^h$ satisfies $|m_h-\gamma(s)|=|\gamma(s+h)-m_h|$
for $h>0$ and $s \in \R$, where we indicate the dependence of matching point
and location by the index $h$.
\end{definition}

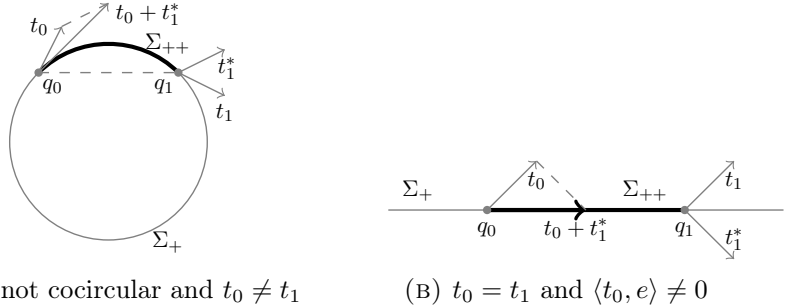
\begin{figure}[h]
\centering
\begin{subfigure}[t]{0.33\textwidth}
\centering
	\scalebox{1.3}{
	  \begin{tikzpicture}
		\draw[gray] (0cm,0cm) circle [radius=1cm];
		\draw[very thick] (0.7071,0.7071) arc (45:135:1);
		\filldraw[gray] (45:1cm) circle(1pt);
		\filldraw[gray] (135:1cm) circle(1pt);
		\draw (45:0.8cm) node[scale=0.6] {$q_1$};
		\draw (135:0.8cm) node[scale=0.6] {$q_0$};
		\draw[dashed,gray] (135:1cm) -- (45:1cm);
		\draw[->,gray] (45:1cm) -- (1.1785,0.4714);
		\draw (1.1785,0.3) node[scale=0.6] {$t_1$};
		\draw[->,gray] (135:1cm) -- (-0.4714,1.1785);
		\draw (-0.7,1.1785) node[scale=0.6] {$t_0$};
		\draw[->,gray] (45:1cm) -- (1.1785,0.9428);
		\draw (1.2,0.75) node[scale=0.6] {$t_1^*$};
		\draw[->,gray] (135:1cm) -- (0, 1.4142);
		\draw (0.4,1.3) node[scale=0.6] {$t_0+t_1^*$};
		\draw[dashed,gray] (-0.4714,1.1785) -- (0,1.4142);
		\draw (0.6,1) node[scale=0.6] {$\Sigma_{++}$};
		\draw (0.6,-1.02) node[scale=0.6] {$\Sigma_{+}$};
 	 \end{tikzpicture} 
 	 }
 	 \caption{not cocircular and $t_0\neq t_1$}
\end{subfigure}
	\quad
	\quad
\begin{subfigure}[t]{0.33\textwidth}
\centering
	\scalebox{1.3}{
	  \begin{tikzpicture}
	  \draw[very thick, ->] (0,0) -- (1,0);
	  \draw[very thick] (0,0) -- (2,0);
		\filldraw[gray] (0,0) circle(1pt);
		\filldraw[gray] (2,0) circle(1pt);
		\draw[-,gray] (-1,0) -- (0,0);
		\draw[-,gray] (2,0) -- (3,0);
		\draw (-0.7,0.2) node[scale=0.6] {$\Sigma_{+}$};
		\draw (1.6,0.2) node[scale=0.6] {$\Sigma_{++}$};
		\draw (0,-0.2) node[scale=0.6] {$q_0$};
		\draw (2,-0.2) node[scale=0.6] {$q_1$};
		\draw[->,gray] (0,0) -- (0.5,0.5);
		\draw (0.5,0.3) node[scale=0.6] {$t_0$};
		\draw[->,gray] (2,0) -- (2.5,0.5);
		\draw (2.5,0.3) node[scale=0.6] {$t_1$};
		\draw[->,gray] (2,0) -- (2.5,-0.5);
		\draw (2.5,-0.3) node[scale=0.6] {$t_1^*$};
		\draw[dashed,gray] (0.5,0.5) -- (1,0);
		\draw (0.9,-0.2) node[scale=0.6] {$t_0 + t_1^*$};
 	 \end{tikzpicture}
 	 } 
	\caption{$t_0 = t_1$ and $\langle t_0, e \rangle \neq 0$}
\end{subfigure}	
\caption{The set $\Sigma_{++}$ (boldface) contained in $\Sigma_+$}
\label{sigma_plus_plus}
\end{figure}
 Item (iv) of Definition \ref{sigma++} requires that the matching point
$m_h$ bisects the segment connecting $\g(s)$ and $\g(s+h)$. That this
is indeed possible for sufficiently small $h$ is the content of
the following result. Note that here and throughout the paper
we use the
\emph{periodic norm} 
\begin{equation}\label{eq:periodic-norm}
|s-t|_{\Sl}:=\min_{k \in \Z} |s+Lk-t|
\end{equation}
to measure distances in the periodic domain $\R/L\Z$. Moreover, 
for a continuous function $f$ on $[0,L]$ we
denote by $\omega_f:[0,L]\to [0,\infty)$ its \emph{modulus of 
continuity}, 
which satisfies $\omega_f(0)=0$
and which can be chosen to 
be concave and non-decreasing.

\begin{lemma}[Existence of $\g$-interpolating proper biarc]
\label{existenz_biarc_kurve}
Let $\gamma \in C^1_\textnormal{ia}\left(\Sl,\R^3\right)$  and
$h\in (0,\frac{L}2]$ such that $\wg(h)<\frac{1}{2}$
Then there exists a proper $\gamma$-inter\-polating balanced biarc 
interpolating the point-tangent pair
$
\left(\left[\gamma(s),\gamma'(s)\right],\left[\gamma(s+h),
\gamma'(s+h)\right]\right)$ for all $s\in\R.
$
\begin{proof} 
First notice that \eqref{bruch_s} and \eqref{bruch_s+h}
of Lemma \ref{absch_sk} together with $\wg(h)<\frac{1}{2}$ and
the injectivity of $\g$
imply that 
\begin{equation}
\label{groesser0}
\left\langle \gamma'(s), \gamma(s+h) - \gamma(s) \right\rangle >0\quad
\AND\quad
\left\langle \gamma'(s+h), \gamma(s+h) - \gamma(s) \right\rangle >0.
\end{equation}
Thus, the point-tangent pair is proper according to Definition
\ref{sigma++} (ii).

If the point-tangent pair is not incompatible cocircular and 
$\gamma'(s)\neq \gamma'(s+h)$ holds, it follows from 
Proposition \ref{existence_biarcs} (i) and (ii) (a) that 
$\Sigma_+^h$ is the circle $\mathcal{C}_{+}^{'}$.
Hence, $\Sigma_{++}^h$ is a circular arc between $\gamma(s)$ and $\gamma(s+h)$. Thus, the matching point $m_h$ can be chosen in $\Sigma_{++}^h$ such that $\left|m_h - \gamma(s)\right|=\left| \gamma(s+h)- m_h\right|$
holds.

If $\gamma'(s)=\gamma'(s+h)$ holds, $\Sigma_{+}^h$ is a straight line 
as a consequence of Proposition \ref{existence_biarcs} (iii), 
since we obtain $\left\langle \gamma'(s), e_h \right\rangle>0$
by dividing \eqref{groesser0} through $|\gamma(s+h)-\gamma(s)|$,
thus excluding case (iv) of Proposition \ref{existence_biarcs}.
Moreover, with  $\gamma'(s)=\gamma'(s+h)$ we infer for the
unit vector $e_h:=\g(s+h)-\g(s)/|\g(s+h)-\g(s)|$  by means 
of \eqref{rot_matrix} and \eqref{eq:t0*t1*}
\begin{align*}
\gamma'(s)+ \left(\gamma'(s+h) \right)^*\,\,\,\beq[\eqref{rot_matrix},\eqref{eq:t0*t1*}]\,\,\gamma'(s) + 2 \left(e_h \otimes e_h\right) \gamma'(s+h) - \gamma'(s+h)
= 2 \left\langle \gamma'(s), e_h \right\rangle e_h.
\end{align*}
Thus, the vector $\gamma'(s)+ \left(\gamma'(s+h) \right)^*$ is a 
positive multiple  of the vector $e_h$. 
In particular, $\gamma'(s)+ \left(\gamma'(s+h) \right)^*$ 
has the same orientation as $e_h$. According to Definition \ref{sigma++}, $\Sigma_{++}^h$ is in this case the line segment between $\gamma(s)$ and $\gamma(s+h)$. Therefore, the matching point $m_h$ in $\Sigma_{++}^h$ can be also chosen such that $\left| m_h-\gamma(s)\right| = \left| \gamma(s+h)-m_h\right|$.

So, we are finished with the proof once we have shown that the smallness
condition on $h$ excludes
case (ii) (a) of Proposition \ref{existence_biarcs}.
Indeed, suppose that the point-tangent pair was incompatible cocircular.
Then 
\begin{align}
\label{tangenten_gleichheit}
\gamma'(s) + \left(\gamma'(s+h) \right)^*=0,
\end{align}
and using \eqref{eq:t0*t1*}
we can write 
\begin{align*}
\left(\gamma'(s+h) \right)^* = 2 e_h \otimes e_h \gamma'(s+h) - \gamma'(s+h)
= 2 \left\langle e_h, \gamma'(s+h) \right\rangle e_h  - \gamma'(s+h).
\end{align*}
This representation inserted in \eqref{tangenten_gleichheit} 
leads to
$
\gamma'(s+h)-\gamma'(s) =  2 \left\langle e_h, \gamma'(s+h) \right\rangle e_h
$
and hence,
\begin{align}
\label{gleichheit}
\textstyle
\left| \gamma'(s+h) - \gamma'(s) \right|=2 \left| \left\langle e_h, \gamma'(s+h) \right\rangle \right| \underbrace{\left| e_h \right|}_{=1}= 2 \frac{\left\langle \gamma(s+h) - \gamma(s),\gamma'(s+h) \right\rangle}{\left|\gamma(s+h) - \gamma(s) \right|}.
\end{align}
By virtue of inequality  \eqref{bruch_s+h} in 
Lemma \ref{absch_sk} we conclude
\begin{align*}
\textstyle
2\left(1-\wg(h)\right) \bleq[\eqref{bruch_s+h}] 2 \frac{\left\langle \gamma(s+h) - \gamma(s),\gamma'(s+h) \right\rangle}{\left|\gamma(s+h) - \gamma(s) \right|}\beq[\eqref{gleichheit}] \left| \gamma'(s+h) - \gamma'(s) \right| \leq \wg(h),
\end{align*}
which is equivalent to $\wg(h)\geq \frac{2}{3}$
contradicting our assumption on $h$. 
\end{proof}
\end{lemma}

Glueing together finitely many interpolating biarcs in a $C^1$-fashion
produces 
biarc curves precisely defined as follows.
\begin{definition}\cite[cf. Definition 6.1]{smutny_2004}
\label{biarc_Kurve}
{\rm (i)}\,
A closed \emph{biarc curve} $\beta : J \to \R^3$ is a closed curve 
assembled from biarcs in a $C^1$-fashion where the biarcs 
interpolate a sequence $\left(\left[q_i,t_i\right]\right)_{i \in I}$ 
of point-tangent tuples. 
$J$ is a compact interval, $I \subset \N$ bounded, and the first and 
last point-tangent tuple  coincide. The set of such biarc curves
is denoted by $\tilde{\B}_n$ where $n$ is the number of
indices contained in $I$.

{\rm (ii)}\,
We call a closed biarc-curve \emph{proper} if every biarc of the curve is proper.

{\rm (iii)}\,
A biarc-curve is 
\emph{$\gamma$-interpolating} for a given curve
$\g\in C^1_\textnormal{ia}(\R/L\Z,\R^3)$,
and \emph{balanced} if every biarc 
of the curve is $\gamma$-interpolating and balanced.

\end{definition}
Notice that the set $\B_n$ of closed biarc curves satisfying
\eqref{eq:biarc-lengths} introduced in the introduction
is a strict subset of $\tilde{\B}_n$.

Under suitable control of partitions of the periodic domain
we can prove the existence of proper, $\g$-interpolating, and
balanced biarc curves in Lemma \ref{existenz_biarc_kurve_folge} below. 
\begin{definition}
\label{def_part}
Let $c_1,c_2 >0 $. A sequence $\left(\mathcal{M}_n\right)_{n \in \N}$  of partitions of $\Sl$ with \linebreak $\mathcal{M}_n:=\left\{s_{n,0}, \dots, s_{n,n}\right\}$ and $0=\sn{0} < \sn{1} < \dots < \sn{n-1}<\sn{n}=L$ is called
\emph{$(c_1-c_2)$-distributed} if and only if for
\begin{align}
h_n:=\max_{k=0, \dots, n-1} \diffsn{k} \quad \text{and} \quad \tilde{h}_n:=\min_{k=0, \dots, n-1} \diffsn{k}\label{part}
\end{align}
one has
\begin{align}
\label{prop_part}
\textstyle
\frac{c_1}{n}\leq \tilde{h}_n \leq h_n \leq \frac{c_2}{n} \quad \text{and} \quad h_n \leq \frac{L}{2}\quad\Foa n\in\N.
\end{align}
\end{definition}

\begin{lemma}
\label{existenz_biarc_kurve_folge}
Let $\gamma \in C^1_\textnormal{ia}\left(\Sl,\R^3\right)$  $c_1,c_2 >0$ 
and $\left(\mathcal{M}_n\right)_{n \in \N}$ a sequence of 
$(c_1-c_2)$-distributed partitions. Then there is some
$N\in\N$ such that for all $n\ge N$ there exists a proper 
$\gamma$-interpolating and balanced biarc-curve $\beta_n$ interpolating 
the point-tangent pairs
\begin{align*}
\big(\left(\left[\gamma\left(\sn{i}\right),\gamma'\left(\sn{i}\right)\right],\left[\gamma\left(\sn{i+1}\right),\gamma'\left(\sn{i+1}\right)\right]\right)\big)_{i=0, \dots, n-1}.
\end{align*}
\begin{proof}
By means of the defining inequality \eqref{prop_part} for the 
$(c_1-c_2)$-distributed sequence $\left(\mathcal{M}_n\right)_{n \in \N}$
we have
 $|\sn{i+1}-\sn{i}|_{\Sl}=|\sn{i+1}-\sn{i}|$ 
for all $n\in\N$ and $i=0, \dots, n-1$
(see \eqref{eq:periodic-norm}), and we can choose $N\in\N$ so large
that 
the inequalities $\wg\left( \frac{c_2}{N}\right) < \frac{1}{2}$ 
and $\frac{c_2}{N} \leq \frac{L}{2}$ hold.  Then, in particular,
\begin{align*}
\textstyle
\wg\left(\diffsn{i} \right)\leq \wg(h_n)\bleq[\eqref{prop_part}] \wg\left( \frac{c_2}{N}\right)< \frac{1}{2}\quad\Foa n\ge N.
\end{align*}
As a consequence of Lemma \ref{existenz_biarc_kurve}, there exists for 
all $n\ge N$ and $i=0, \dots, n-1$ 
a proper $\gamma$-interpolating and balanced biarc interpolating the 
point-tangent pair 
$\left(\left[\gamma\left(\sn{i}\right),\gamma'
\left(\sn{i}\right)\right],\left[\gamma\left(\sn{i+1}\right),\gamma'\left(\sn{i+1}\right)\right]\right)$.
Now, we assemble for $i=0, \dots, n-1$ these $n$ biarcs as in Definition \ref{biarc_Kurve} and obtain a biarc-curve with the required properties.
\end{proof}
\end{lemma}

From now on, whenever we write $\beta_n$
for a given curve $\g\in C^1_\textnormal{ia}(\R/L\Z,\R^3)$, 
we mean a proper $\gamma$-interpolating and balanced biarc-curve obtained in Lemma \ref{existenz_biarc_kurve_folge}. By $\lambda_{n,i}$ we denote the length of the $i$-th biarc of the curve $\beta_n$.
In general, the elements $\beta_n$ do not have the same length 
as the interpolated curve $\gamma$. However, Smutny showed in 
\cite{smutny_2004} that under certain assumptions 
the sequence of the lengths 
$\left(\mathscr{L}\left(\beta_n\right)\right)_{n \in \N}$ of $\beta_n$ 
converges towards the length $\mathscr{L}(\g)$ of $\gamma$. 
The following lemma is an essential ingredient for that proof.
\begin{lemma}
\label{Konv_max_bruch_lamda}
Let $\gamma \in C^{1,1}_\textnormal{ia}(\Sl,\R^3)$ and 
$\left(\beta_n\right)_{n \in \N}$ be a sequence of proper 
$\gamma$-inter\-polating and balanced biarc curves as in Lemma 
\ref{existenz_biarc_kurve_folge}. Then
\begin{align*}
\textstyle\max_{i=0, \dots, n-1} \Big|
\frac{\lambda_{n,i}}{\diffsn{i}}-1\Big| \to 0 
\quad\text{ as } n \to \infty.
\end{align*}
\begin{proof}
We identify the periodic domain $\R/L\Z$ with $[0,L]$
and check that $(c_1-c_2)$-distributed partitions
of $\R/L\Z$ satisfy Smutny's requirements in
\cite[Notation 6.2, 6.3]{smutny_2004} apart from nestedness of the
mesh. The latter, however, is not necessary in her proof; whence
we can apply 
\cite[Lemma 6.8]{smutny_2004} to conclude.
\end{proof}
\end{lemma}

Now we show that the lengths $\mathscr{L}(\beta_n)$
of 
proper $\gamma$-interpolating and balanced biarc curves $\beta_n$
converge towards the length $\mathscr{L}(\g)$ of $\g$.

\begin{theorem}
\label{Laengenkonvergenz}
Let $\gamma \in C^{1,1}_\textnormal{ia}\left(\Sl,\R^3\right)$ 
and $\left(\beta_n\right)_{n \in \N}$ be a sequence of proper 
$\gamma$-interpolating and balanced biarc curves.  Then
$
\big|\frac{\mathscr{L}\left(\beta_n\right)}{\mathscr{L}(\gamma)}
-1\big| \to 0$  as  $n \to \infty.
$
\begin{proof}
This follows straight from \cite[Corollary 6.9]{smutny_2004}; 
under the same preconditions as we verified in the proof of Lemma 
\ref{Konv_max_bruch_lamda}.
\end{proof}
\end{theorem}

In order to address convergence of biarc curves $\beta_n$
to the
interpolated curve $\g$ we need to reparametrize $\beta_n$ for all 
$n \in \N$ such that those reparametrizations 
are defined on $\Sl$ like $\gamma$ is. 
An explicit reparametrization function which maps the arclength parameters of $\gamma$ at the supporting points of the mesh to the arclength parameters of $\beta_n$ is constructed in \cite[Appendix A]{smutny_2004}.
With that we can show the 
$C^1$-convergence of a reparametrized sequence of biarc curves to the 
interpolated curve 
$\gamma$.

\begin{theorem}
\label{C1_Konvergenz}
Let $\gamma \in C^{1,1}_\textnormal{ia}
\left(\Sl,\R^3\right)$, and let $\left(\beta_n\right)_{n \in \N}$ be a 
sequence of proper $\gamma$-interpolating and balanced biarc curves  
parametrized by arclength. Then for $B_n:=\beta_n \circ \varphi_n$ 
with $\varphi_n$ as constructed in \cite[Appendix A]{smutny_2004} 
one has
$
\|\gamma-B_n\|_{C^1} \to 0$ as $n \to \infty.$
\begin{proof}
We want to apply \cite[Theorem 6.13]{smutny_2004}, where Smutny showed 
$C^1$-convergence under certain assumptions. Additionally to the 
hypotheses checked before in the proof of
Lemma \ref{Konv_max_bruch_lamda}, we need to show that the so called 
biarc parameters $\Lambda_{n,i}$ of the $i$-th biarc of the biarc 
curve $\beta_n$, representable as (cf.
\cite[Lemma 4.13]{smutny_2004})
\begin{align*}
\textstyle
\Lambda_{n,i}=\frac{\langle \gamma'(\sn{i}),\gamma(\sn{i+1})-\gamma(\sn{i})\rangle |m_{n,i}-\gamma(\sn{i})|^2}{\langle \gamma'(\sn{i}), m_{n,i} - \gamma'(\sn{i}) \rangle |\gamma(\sn{i+1})-\gamma(\sn{i})|^2},
\end{align*}
where the $m_{n,i}$ are  the matching points of the $i$-th biarc,
are uniformly bounded from below and from above.
In other words, we have to prove that 
there exist two constants $\Lambda_{\min},\Lambda_{\max}$ such that
\begin{align*}
0 < \Lambda_{\min} \leq \Lambda_{n,i} \leq \Lambda_{\max} <1
\quad\Foa n \in \N,\,i=0,\dots, n-1.
\end{align*}
Using the fact, that the biarc curves are balanced, i.e.
$
|m_{n,i}-\gamma(\sn{i})|=|m_{n,i}-\gamma(\sn{i+1})|,
$
and that $\gamma$ is parametrized by arclength, we can then estimate
by means of \eqref{bruch_s} in Lemma \ref{absch_sk} in the appendix
\begin{align*}
\Lambda_{n,i}&\bgeq \textstyle
\frac{\langle \gamma'(\sn{i}),\gamma(\sn{i+1})-\gamma(\sn{i})\rangle|m_{n,i}-\gamma(\sn{i})|^2}{|m_{n,i}-\gamma(\sn{i})\|\gamma(\sn{i+1})-\gamma(\sn{i})|^2}
\beq\textstyle
\frac{\langle \gamma'(\sn{i}),\gamma(\sn{i+1})-\gamma(\sn{i})\rangle}{|\gamma(\sn{i+1})-\gamma(\sn{i})|}\frac{|m_{n,i}-\gamma(\sn{i})|}{|\gamma(\sn{i+1})-\gamma(\sn{i})|}\\
&\bgeq \textstyle
\frac{\langle \gamma'(\sn{i}),\gamma(\sn{i+1})-\gamma(\sn{i})\rangle}{|\gamma(\sn{i+1})-\gamma(\sn{i})|}\underbrace{\textstyle
\frac{|m_{n,i}-\gamma(\sn{i})|}{|m_{n,i}-\gamma(\sn{i+1})|+|m_{n,i}-\gamma(\sn{i})|}}_{=\frac{1}{2}}\\
&\overset{\eqref{bruch_s}}{\ge}\textstyle
\frac{1}{2}(1-\wg(|s_{n,i+1}-s_{n,i}|))
\ge\frac{1}{2}(1-\wg(h_{n})).
\end{align*}
Hence, we can choose $n$ sufficiently large, such that $\frac{1}{2}(1-\wg(h_{n}))\geq \frac{1}{4}=:\Lambda_{\min}$. On the other hand, by \cite[Lemma 5.6]{smutny_2004} we have
\begin{align*}
\Lambda_{n,i}=1-\bar{\Lambda}_{n,i}+O(h_{n,i}^2)\quad
\Foa i=0,\ldots,n-1,\,\textnormal{as $n\to\infty$.}
\end{align*}
where the constant hidden in the $O(h_{n,i}^2)$-term
only depends on the curve $\g$ and
where $\bar{\Lambda}_{n,i}$ is given by
\begin{align*}
\textstyle
\bar{\Lambda}_{n,i}
=\frac{\langle \gamma'(\sn{i+1}),\gamma(\sn{i+1})-\gamma(\sn{i})\rangle |m_{n,i}-\gamma(\sn{i+1})|^2}{\langle \gamma'(\sn{i+1}), \gamma'(\sn{i+1})-m_{n,i} \rangle |\gamma(\sn{i+1})-\gamma(\sn{i})|^2}.
\end{align*}
As for $\Lambda_{n,i}$, we can estimate 
$\bar{\Lambda}_{n,i}\geq \frac{1}{2}(1-\wg(h_{n}))$, which yields
\begin{align*}
\Lambda_{n,i}&\textstyle
\leq 1-\frac{1}{2}(1-\wg(h_{n})) + O(h_n^2)
\leq\frac{1}{2}(1+\wg(h_n))+O(h_n^2)
\quad\textnormal{as $n\to\infty$.}
\end{align*}
Hence, again choosing $n$ sufficiently large, we obtain 
$\frac{1}{2}(1+\wg(h_n))+O(h_n^2)\leq \frac{3}{4}=:\Lambda_{\max}$ 
as the necessary uniform bound on the biarc parameters. 
Therefore, \cite[Theorem 6.13]{smutny_2004} is applicable and we obtain that 
$B_n \to \gamma$ in $C^1$ as $n \to \infty$.
\end{proof}
\end{theorem}

\section{Discrete energies on interpolating biarc curves converge
to the continuous $\TP_q$-energy}
\label{approx_energies}
For
the central convergence result of this section, Theorem
\ref{konv_ordnung_energien}, we work with discrete tangent-point
energies $\tilde{\E^n_q}$ with the larger effective domain $\tilde{\B}_n$
(see Definition \ref{biarc_Kurve} (i)),
instead of with 
$\E^n_q$ introduced in \eqref{eq:discrete-tan-point}  of
the introduction, whose effective domain $\B_n$ is defined
by the constraint 
\eqref{eq:biarc-lengths}. 
In other words,
\begin{align}\label{eq:larger-discrete-tan-point}
\tilde{\E}_q^n(\gamma):=\begin{cases}
\displaystyle \sum_{i=0}^{n-1} \sum_{j=0,j\neq i}^{n-1} 
\left(\frac{2\;\dist(l(q_j),q_i)}{|q_i-q_j|^2}\right)^q \lambda_i \lambda_j, & \text{if }\gamma \in \tilde{\mathcal{B}}_n\\
\infty, & \, \text{otherwise.}
\end{cases}
\end{align}
These discrete energies evaluated on a sequence 
$\left(\beta_n\right)_{n \in \N}$ of proper $\g$-interpolating and
balanced biarc curves converge with a certain rate
to the continuous $\TP_q$-energy of $\g$ if
$\g$ is sufficiently smooth.
Some of the ideas in the proof of the theorem are based on 
\cite[Proposition 3.1]{scholtes_2014b} by Scholtes.

\begin{theorem}
\label{konv_ordnung_energien}
Let $c_1,c_2>0$ and $\gamma \in C^{1,1}_\textnormal{ia}
\left(\Sl,\R^3\right)$, and  
$\left(\mathcal{M}_n\right)_{n \in \N}$ with 
$\mathcal{M}_n=\left\{\sn{0},\dots,\sn{n} \right\}$ be a 
$(c_1-c_2)$-distributed sequence of partitions of $\Sl$ (see
Definition \ref{def_part}). 
Let $\left(\beta_n\right)_{n \in \N}$, with $\beta_n \in \tilde{\mathcal{B}}_n$ for all $n \in \N$, be a sequence of proper $\gamma$-interpolating and balanced biarc curves interpolating the point-tangent
data 
$$\left(\left(\left[\gsn{i},\gstrichsn{i}\right],\left[\gsn{i+1},\gstrichsn{i+1}\right]\right)\right)_{i=0, \dots, n-1}.
$$
Then there exists an $N\in \N$ such that for every 
$\varepsilon>0$ there is a constant $C_\varepsilon>0$ such that
\begin{align*}
\big|\TP_q\left(\gamma\right)
-\tilde{\E}_q^n\left(\beta_n\right)\big|\leq \textstyle
\frac{C_\varepsilon}{n^{1-\varepsilon}}\quad\Foa n\ge N.
\end{align*}
\begin{proof}
Set $\Upsilon=4\frac{c_2}{c_1}$ and define for $i,j \in \{0, \dots, n\}$
the periodic index distance 
\begin{align*}
|i-j|_n:=\min \{|i-j|,n-|i-j|\}.
\end{align*}
We then decompose 
\begin{align}
\label{zerl_energien}
\begin{split}
\TP_q\left(\gamma\right)-\tilde{\E}_q^n\left(\beta_n\right)
=&\textstyle\sum_{i=0}^{n-1}\sum_{j, |i-j|_n\leq \Upsilon} \int_{\sn{j}}^{\sn{j+1}} \int_{\sn{i}}^{\sn{i+1}} \left(2\frac{\distg{s}{t}}{\diffgamma{s}{t}^2}\right)^q \d s \d t\\
&\textstyle
-\sum_{i=0}^{n-1}\sum_{j, 0<|i-j|_n\leq \Upsilon} \left(2\frac{\distg{\sn{i}}{\sn{j}}}{\diffgamma{\sn{i}}{\sn{j}}^2}\right)^q \lambda_{n,i}\lambda_{n,j}\\
&\textstyle
+2^q \sum_{i=0}^{n-1}\sum_{j, |i-j|_n> \Upsilon} \left(A_{i,j}+B_{i,j}+C_{i,j}\right),
\end{split}
\end{align}
with
\begin{align*}
A_{i,j}&\textstyle:=\int_{\sn{j}}^{\sn{j+1}} \int_{\sn{i}}^{\sn{i+1}}\frac{\distg{s}{t}^q-\distg{\sn{i}}{\sn{j}}^q}{\diffgamma{s}{t}^{2q}}\d s \d t, \\
B_{i,j}&\textstyle:=\int_{\sn{j}}^{\sn{j+1}} \int_{\sn{i}}^{\sn{i+1}} \distg{\sn{i}}{\sn{j}}^q\left( \frac{1}{\diffgamma{s}{t}^{2q}}-\frac{1}{\diffgamma{\sn{i}}{\sn{j}}^{2q}}\right) \d s \d t, \\
C_{i,j}&\textstyle:=\frac{\distg{\sn{i}}{\sn{j}}^q}{\diffgamma{\sn{i}}{\sn{j}}^{2q}}\big[\diffsn{i}\diffsn{j}-\lambda_{n,i}\lambda_{n,j} \big].
\end{align*}
\textbf{Step 1:} 
Since $\g$ is an injective $C^1$-curve it is
bilipschitz (see Lemma \ref{lem:bilipschitz}), i.e., 
there exists a constant $c_\gamma \in (0,\infty)$ such that 
\begin{align}
\label{bi_lipschitz}
|t-s|_{\Sl}\leq c_\gamma |\gamma(t)-\gamma(s)|\quad\Foa t, s \in \R.
\end{align}

\textbf{Step 2:} 
Now, we give an upper bound for $2\frac{\distg{s}{t}}{\diffgamma{s}{t}^2}$
for all $s,t \in \R$ with $s\neq t$.
Without loss of generality we assume $t<s$. 
Then there exists a number $k=k(s,t) \in \Z$ 
satisfying $|t-s|_{\Sl}=|kL+t-s|$. 
We use 
the periodicity of $\gamma$   
and $K:=\|\g''\|_{L^\infty}<\infty$ 
(since $\g\in C^{1,1}(\R/L\Z,\R^3)\simeq W^{1,\infty}(\R/L\Z,\R^3)$)
to estimate
\begin{align}
&\distg{s}{t}=\inf_{\mu \in \R}|\gamma(s)-\gamma(t)-\mu\gamma'(t)|
\notag\\
&\le|\g(s)-\gamma(kL+t)-(s-(kL+t))\gamma'(kL+t)|\notag\\
&\le \textstyle\int^s_{kL+t} \int^u_{kL+t} |\gamma''(v)| \d v \d u
\le K (kL+t-s)^2=K|t-s|_{\Sl}^2,\label{absch_dist_nach oben}
\end{align}
where we assumed without loss of generality that $kL+t<s$ for the
integrals.
Therefore,  by means of \eqref{bi_lipschitz}
\begin{align} 
\textstyle
\left(2\frac{\distg{s}{t}}{\diffgamma{s}{t}^2}\right)^q &\leq (2 c_\gamma^2 K)^q \label{absch_bruch}\quad\Foa s,t\in\R,\, s\not=t.
\end{align}
Define $C_{1}:= \left(2 c_\gamma^2 K\right)^q c_2 (2\Upsilon+1)L$. 
Applying the calculations above we can estimate the first term
on the right-hand side of  \eqref{zerl_energien} from above by
\begin{align}
\label{Term1}
\bleq[\eqref{absch_bruch}]& 
\textstyle
(2 c_\gamma^2 K)^q \sum_{i=0}^{n-1}\sum_{j, |i-j|_n\leq \Upsilon}\diffsn{i}\diffsn{j}\nonumber\\
\bleq[\eqref{prop_part}]& \textstyle
(2 c_\gamma^2 K)^q \frac{c_2}{n} \sum_{i=0}^{n-1} 
\diffsn{i}\underbrace{\textstyle
\sum_{j, |i-j|_n\leq \Upsilon} 1}_{\leq  2\Upsilon+1}\nonumber\\
\bleq& 
(2 c_\gamma^2 K)^q {\textstyle
\frac{c_2}{n}} (2\Upsilon+1) \underbrace{ 
\textstyle\sum_{i=0}^{n-1} \diffsn{i}}_{=L}
\beq
\frac{C_{1}}{n}.
\end{align}

\textbf{Step 3:} By Lemma \ref{Konv_max_bruch_lamda} we have
$
\max_{k=0, \dots, n-1} \frac{\lambda_{n,k}}{\diffsn{k}} \to 1$ as
$n \to \infty.
$
Hence, there exists a constant $c_\lambda \in (1,\infty)$ such that
\begin{align}
\lambda_{n,i} \leq c_\lambda\diffsn{i}\quad\Fo i=0, \dots, n-1.
\label{lambda}
\end{align}
Define $C_{2}:=c_\lambda^2 C_{1}$. Thus,
\begin{align}
\label{Term2}
&\textstyle
\sum_{i=0}^{n-1}\sum_{j, 0<|i-j|_n\leq \Upsilon} \left(2\frac{\distg{\sn{i}}{\sn{j}}}{\diffgamma{\sn{i}}{\sn{j}}^2}\right)^q \lambda_{n,i}\lambda_{n,j}\nonumber\\
\bleq[\eqref{lambda}]& \textstyle
c_\lambda^2 \sum_{i=0}^{n-1}\sum_{j, 0<|i-j|_n\leq \Upsilon} \left(2\frac{\distg{\sn{i}}{\sn{j}}}{\diffgamma{\sn{i}}{\sn{j}}^2}\right)^q \diffsn{i}\diffsn{j}\nonumber\\
\bleq[\eqref{absch_bruch}]& {\textstyle
c_\lambda^2 (2c_\gamma^2K)^q  \sum_{i=0}^{n-1}\sum_{j, 0<|i-j|_n\leq \Upsilon} \diffsn{i}\diffsn{j}}
\overset{\eqref{Term1}}{\le} \frac{C_{2}}{n},
\end{align}
which takes care of the second term on the right-hand side
of \eqref{zerl_energien}.

\textbf{Step 4:} The sequence $(\mathcal{M}_n)_n$ is assumed
to be 
$(c_1-c_2)$-distributed, so that in view of \eqref{prop_part}
\begin{align*}
\diffsn{k}=\diffsn{k}_{\Sl} \quad\Foa n\in\N\AND k=0, \dots, n-1.
\end{align*}
For $s \in [\sn{i},\sn{i+1})$ and $t \in [\sn{j},\sn{j+1})$ 
with $i\neq j$ we use $|s-t|_{\R/L\Z}\le |s-\sn{i}|_{\Sl}
+ |\sn{i}-\sn{j}|_{\Sl}+|\sn{j}-t|_{\Sl}$ to infer
the inequality
\begin{align*}
|t-s|_{\Sl}
&\textstyle
\bleq \left|\sn{i}-\sn{j}\right|_{\Sl}+2\max_{k=0, \dots, n-1} \diffsn{k}_{\Sl}\\
&\textstyle\bleq[\eqref{prop_part}] 
\left|\sn{i}-\sn{j}\right|_{\Sl} + 2\frac{c_2}{c_1} 
\underbrace{\textstyle\min_{k=0, \dots, n-1} 
\diffsn{k}_{\Sl}}_{\leq \left|s_{n,i}-s_{n,j}\right|_{\Sl}} \\
&\textstyle
\bleq \big(1+2\frac{c_2}{c_1}\big) \left|\sn{i}-\sn{j}\right|_{\Sl}.
\end{align*}
Now let $|i-j|_n>\Upsilon =4\frac{c_2}{c_1}$. In particular,
\begin{align}
\textstyle\label{ij}
2\frac{c_2}{c_1}<\frac{1}{2}|i-j|_n.
\end{align}
Then, similarly as before,
\begin{align*}
|t&-s|_{\Sl}
\textstyle\bgeq \left|\sn{i}-\sn{j}\right|_{\Sl} - 2\max_{k=0, \dots, n-1} \diffsn{k}_{\Sl}\\
&\bgeq[\eqref{prop_part}] |i-j|_n \min_{k=0, \dots, n-1} 
\diffsn{k}_{\Sl}-{\textstyle\frac{2c_2}{c_1}}\min_{k=0, \dots, n-1} \diffsn{k}_{\Sl}\\
&\textstyle\beq\big(|i-j|_n-\frac{2c_2}{c_1}\big)
\min_{k=0, \dots, n-1} \diffsn{k}_{\Sl}\\
&\textstyle
\overset{\eqref{ij}}{>} \frac{1}{2} |i-j|_n\min_{k=0, \dots, n-1} \diffsn{k}_{\Sl}
\textstyle\bgeq\frac{c_1}{2c_2}\left|\sn{i}-\sn{j}\right|_{\Sl}.
\end{align*}
In total, we conclude for $|i-j|_n>\Upsilon $ 
\begin{align}
\label{absch_ij}
\textstyle\frac{c_1}{2c_2}\left|\sn{i}-\sn{j}\right|_{\Sl} 
\leq |t-s|_{\Sl} \leq \big(1+2\frac{c_2}{c_1}\big)
\left|\sn{i}-\sn{j}\right|_{\Sl}.
\end{align}
for $s\in [\sn{i},\sn{i+1})$ and $t\in [\sn{j},\sn{j+1})$,
which we consider also in Steps 5 and 6. In addition, we assume
from now on that $|i-j|_n>\Upsilon $.

\textbf{Step 5:} In order to estimate $A_{i,j}$, we initially 
estimate for arbitrary $a,b\geq 0$
\begin{align}
\label{a_b_hoch_q}
|b^q-a^q|=\textstyle
\left|\int_a^b \frac{d}{dx}x^q \d x \right|=\left|\int_a^b qx^{q-1} \d x\right|\leq q |b-a|\max\{a,b\}^{q-1},
\end{align}
since the function $f : [0,\infty) \to [0,\infty), x \to x^{q-1}$ is 
non-decreasing for $q\geq 2$.
We abbreviate $d(\cdot,\cdot):=\dist(l(\cdot),\g(\cdot))$
and use estimate \eqref{a_b_hoch_q} to find for 
$s \in [\sn{i},\sn{i+1})$ and $t \in [\sn{j},\sn{j+1})$
\begin{align}
\label{dist_hoch_q}
|d^q(t,s)-d^q(\sn{j},\sn{i})|
\leq q \big|d(t,s)-d(\sn{j},\sn{i})\big|
(\max\left\{d(t,s),d(\sn{j},\sn{i})\right\})^{q-1}.
\end{align}
Furthermore, combining \eqref{absch_dist_nach oben}
with \eqref{absch_ij} yields
\begin{align}
\begin{split}
\label{dist_ij}
d(t,s)&\overset{\eqref{absch_dist_nach oben}}{\leq} K 
|t-s|_{\Sl}^2 \bleq[\eqref{absch_ij}] K 
\textstyle
\big(1+2\frac{c_2}{c_1}\big)^2 \left|\sn{i}-\sn{j}\right|_{\Sl}^2,\\
d(\sn{j},\sn{i})&\leq K \left|\sn{i}-\sn{j}\right|_{\Sl}^2\leq 
\textstyle
K\big(1+2\frac{c_2}{c_1}\big)^2 \left|\sn{i}-\sn{j}\right|_{\Sl}^2.
\end{split}
\end{align}
Hence,
\begin{align}
\label{maxdist}
(\max\left\{d(t,s),d(\sn{j},\sn{i})\right\})^{q-1} 
\leq \textstyle
K^{q-1} \big(1+2\frac{c_2}{c_1}\big)^{2q-2} 
\left|\sn{i}-\sn{j}\right|_{\Sl}^{2q-2}.
\end{align}
Moreover, we estimate again by virtue of \eqref{absch_ij} now
for $s:=\sn{i}$
\begin{align*}
\left|t-\sn{i}\right|_{\Sl}\bleq[\eqref{absch_ij}] \textstyle
\big(1+2\frac{c_2}{c_1}\big)\left|\sn{i}-\sn{j}\right|_{\Sl},
\end{align*}
and we use \eqref{prop_part} to find for $t\in [\sn{j},\sn{j+1})$
\begin{align*}
\left|t-\sn{j}\right|_{\Sl} &\bleq \max_{k=0, \dots, n-1}\diffsn{k}_{\Sl}
\bleq[\eqref{prop_part}]{\textstyle \frac{c_2}{c_1}}
\min_{k=0, \dots, n-1}\diffsn{k}_{\Sl}\\
&\bleq \textstyle
\big(1+2\frac{c_2}{c_1}\big)\left|\sn{i}-\sn{j}\right|_{\Sl}.
\end{align*}
Combining these last two estimates with \eqref{absch_ij} leads to
\begin{align}
\label{absch_max}
\begin{split}
&\max\left\{|t-s|_{\Sl},\left|t-\sn{i}\right|_{\Sl}\right\}+\left|t-\sn{j}\right|_{\Sl}+\left|t-\sn{i}\right|_{\Sl} \\
\leq& 3 
\textstyle
\big(1+2\frac{c_2}{c_1}\big)\left|\sn{i}-\sn{j}\right|_{\Sl}
\quad\Fo s\in [\sn{i},\sn{i+1}),\,t\in [\sn{j},\sn{j+1}).
\end{split}
\end{align}
For arbitrary  $\tau \in \R$ the mapping $P_{\gamma'(\tau)}:
\R^3\to\R\g'(\tau)$ defined as
\begin{align}
P_{\gamma'(\tau)} (v):=\langle v, \gamma'(\tau)\rangle\; \gamma'(\tau),
\quad\Fo v\in\R^3\label{proj}
\end{align}
is the orthogonal projection onto the subspace $\R\gamma'(\tau)$ 
since
$|\gamma'|=1$, and we have
\begin{align}
\left| P_{\gamma'(\tau)} (v) -v \right| &\leq |w-v| 
\quad \text{for all } w \in \R\gamma'(\tau),\,v\in\R^3. \label{proj_dist}
\end{align}
Moreover, we have for any $\tau,\sigma\in\R $
\begin{align}
\label{dist_proj}
d(\tau,\sigma)=\distg{\sigma}{\tau}
&=\left|P_{\gamma'(\tau)} (\gamma(\sigma)-\gamma(\tau))
-(\gamma(\sigma)-\gamma(\tau))\right|.
\end{align}
Furthermore, we calculate for $s\in [\sn{i},\sn{i+1})$ and
$t\in [\sn{j},\sn{j+1})$ using the linearity of the projection
\begin{align}
P_{\gamma'(t)} (\gamma(s) & -\gamma(t))-P_{\gstrichsn{j}} \left(\gsn{i}-\gsn{j}\right) \nonumber\\
\beq & P_{\gamma'(t)} \left(\gamma(s)-\gamma(\sn{i})\right)
+P_{\gamma'(t)}\left(\gamma(\sn{i})-\g(t)\right)\nonumber\\
&-P_{\gstrichsn{j}}\left(\g(t)-\g(\sn{j})\right)
-P_{\gstrichsn{j}}\left(\g(\sn{i})-\g(t)\right)\nonumber\\
\beq[\eqref{proj}] &  P_{\gamma'(t)} \left(\gamma(s)-\gsn{i}\right) - P_{\gstrichsn{j}} \left( \gamma(t)-\gsn{j}\right) \nonumber\\
& +\left\langle \gsn{i} - \gamma(t),\gamma'(t)-\gstrichsn{j} \right\rangle \; \gamma'(t) \nonumber\\
& + \left\langle\gsn{i} - \gamma(t), \gstrichsn{j} \right\rangle \;\left( \gamma'(t) - \gstrichsn{j}\right) \label{proj_st_ij}.
\end{align}
In conclusion, by \eqref{dist_proj} and the elementary inequality
$||a|-|b||\le |a-b|$, this yields for the
expression $|d(t,s)-d(\sn{j},\sn{i})|$
(for $s\in [\sn{i},\sn{i+1})$ and $t\in [\sn{j},\sn{j+1})$) the
upper bound
$$
\left|P_{\gamma'(t)} \left(\gamma(s)\!-\!\gamma(t)\right)-P_{\gstrichsn{j}} \left(\gsn{i}\!-\!\gsn{j}\right)-\left(\gamma(s)\!-\!\gsn{i}\right)+\left(\gamma(t)\!-\!\gsn{j}\right)\right|,
$$
which in turn by means of \eqref{proj_st_ij} and \eqref{proj_dist}
can be bounded from
above by
\begin{align*}
 \big|P_{\gamma'(t)} &\left(\gamma(s)\!-\!\gsn{i}\right)
 \!-\!\left(\gamma(s)\!-\!\gsn{i}\right)\big|\!+\!\left|P_{\gstrichsn{j}} 
 \left( \gamma(t)\!-\!\gsn{j}\right)\!-\!\left(\gamma(t)\!-\!\gsn{j}\right)\right|\\
+\big|\big\langle \g&(\sn{i})\!-\!\gamma(t),\gamma'(t)
\!-\!\gstrichsn{j}\big\rangle\;\gamma'(t)\big|
+\left|\left\langle \gsn{i}\!-\!\gamma(t),\gstrichsn{j}\right\rangle \;(\gamma'(t)\!-\!\gstrichsn{j})\right|  \\
\bleq[\eqref{proj_dist}]& \left|\left(\gamma(s)-\gsn{i}\right)-(s-\sn{i})\gamma'(t)\right| + \left|\left(\gamma(t)-\gsn{j}\right)-(t-\sn{j})\gstrichsn{j}\right|\\
&+2\left|\gsn{i}-\gamma(t)\right|\left|\gstrichsn{j}-\gamma'(t)\right|.
\end{align*}
The last summand  is bounded by $2K|t-\sn{i}|_{\Sl}|t-\sn{j}|_{\Sl}$ since
$K=\|\g''\|_{L^\infty}$ and $1$ are  the Lipschitz constants of $\g'$ 
and $\g$, respectively.
The first summand on the right-hand side above
equals $|\int_{\sn{i}}^s\int_{t}^u \gamma''(v)\d v \d u |$ whereas the
second is bounded by
$\int_{\sn{j}}^t\int_{\sn{j}}^u |\gamma''(v)|\ d v \d u$, so that
we can summarize the estimate
\begin{align}
|d(t,s)&-d(\sn{j},\sn{i})|\bleq K |\sn{i}-s|_{\Sl} 
\max\left\{|s-t|_{\Sl},|t-\sn{i}|_{\Sl}\right\}\nonumber\\
&+ K|t-\sn{j}|_{\Sl}^2 
+2K|t-\sn{i}|_{\Sl}|t-\sn{j}|_{\Sl} \nonumber\\
\bleq& 2K \textstyle\max_{k=0, \dots, n-1} \diffsn{k} \nonumber\\
&\times\Big[\max\left\{|s-t|_{\Sl},|t-\sn{i}|_{\Sl}\right\} 
+ |t-\sn{j}|_{\Sl} + |t-\sn{i}|_{\Sl} \Big] \nonumber \\
\bleq[\eqref{absch_max}]& \textstyle
6 K\big(1+2\frac{c_2}{c_1}\big)
\left|\sn{i}-\sn{j}\right|_{\Sl} \max_{k=0, \dots, n-1} \diffsn{k}
.\label{absch_diff_dist}
\end{align}
Inserting \eqref{maxdist} and \eqref{absch_diff_dist} into 
\eqref{dist_hoch_q}
yields
\begin{equation}
\label{diff_g_final}
|d^q(t,s)-d^q(\sn{j},\sn{i})|\le {\textstyle
6qK^q \big(1+2\frac{c_2}{c_1}\big)^{2q-1}}\left|\sn{i}-\sn{j}\right|_{\Sl}^{2q-1}\max_{k=0, \dots, n-1} \diffsn{k}.
\end{equation}
In order to obtain an estimate for the denominator of $A_{i,j}$ we consider
\begin{align}
\label{nenner}
\textstyle
|\gamma(s)-\gamma(t)|^{2q}\bgeq[\eqref{bi_lipschitz}]
\frac{1}{c_\gamma^{2q}}|t-s|_{\Sl}^{2q} \bgeq[\eqref{absch_ij}] 
\big(\frac{c_1}{2c_2 c_\gamma}\big)^{2q}\left|\sn{i}-\sn{j}\right|_{\Sl}^{2q}.
\end{align}
Setting $C_A:= \frac{c_2^3}{c_1} 
6 q K^q (1+2\frac{c_2}{c_1})^{2q-1} (\frac{2c_2 c_\gamma}{c_1})^{2q}$
we obtain  from \eqref{diff_g_final} and \eqref{nenner}
\begin{align}
\label{aij}
|A_{i,j}|
&\textstyle\bleq[\eqref{diff_g_final},\eqref{nenner}] \quad \frac{c_1}{\left(c_2\right)^3} C_A \int_{\sn{j}}^{\sn{j+1}} \int_{\sn{i}}^{\sn{i+1}}\frac{\left|\sn{i}-\sn{j}\right|_{\Sl}^{2q-1}\underset{k=0, \dots, n-1}{\max}\diffsn{k}}{\left|\sn{i}-\sn{j}\right|_{\Sl}^{2q}}\d s \d t\nonumber\\
&\bleq\textstyle \frac{c_1}{\left(c_2\right)^3} C_A \left(\max_{k=0, \dots, n-1}\diffsn{k}\right)^3\frac{1}{\left|\sn{i}-\sn{j}\right|_{\Sl}} \nonumber\\
&\bleq[\eqref{prop_part}]\textstyle
c_1 C_A \frac{1}{n^3}\frac{1}{|i-j|_n\underset{k=0, \dots, n-1}{\min}\diffsn{k}}
\bleq[\eqref{prop_part}] C_A \frac{1}{n^2} \frac{1}{|i-j|_n}.
\end{align}

\textbf{Step 6:} To estimate $B_{i,j}$, we use \eqref{a_b_hoch_q}
and twice \eqref{absch_ij} leading to
\begin{align}
&\left|\diffgamma{\sn{i}}{\sn{j}}^{2q}-\diffgamma{s}{t}^{2q}\right| \nonumber\\
\beq&\big|\diffgamma{\sn{i}}{\sn{j}}^{q}+\diffgamma{s}{t}^{q}\big|\big|\diffgamma{\sn{i}}{\sn{j}}^{q}-\diffgamma{s}{t}^{q}\big|\nonumber\\
\bleq[\eqref{a_b_hoch_q}]& q\big(|\sn{i}-\sn{j}|_{\Sl}^q+|t-s|_{\Sl}^q\big)\big|\diffgamma{\sn{i}}{\sn{j}}-\diffgamma{s}{t}\big|\nonumber \\
&\times\max\big\{\diffgamma{s}{t},\diffgamma{\sn{i}}{\sn{j}}\big\}^{q-1} \nonumber\\
\bleq[\eqref{absch_ij}]& \textstyle
2q \big(1+2\frac{c_2}{c_1}\big)^q|\sn{i}-\sn{j}|_{\Sl}^q \big|\gsn{i}-\gamma(s)+\gamma(t)-\gsn{j}\big| \nonumber\\
&\times\max\big\{|t-s|_{\Sl},|\sn{i}-\sn{j}|_{\Sl}\big\}^{q-1} \nonumber\\
\bleq[\eqref{absch_ij}]& \textstyle
2q \big(1+2\frac{c_2}{c_1}\big)^{2q-1}|\sn{i}-\sn{j}|_{\Sl}^{2q-1}\big(\diffgamma{s}{\sn{i}}+\diffgamma{t}{\sn{j}}\big) \nonumber \\
\bleq& 2q \textstyle\big(1+2\frac{c_2}{c_1}\big)^{2q-1}
|\sn{i}-\sn{j}|_{\Sl}^{2q-1} \big(|s-\sn{i}|_{\Sl}+|t-\sn{j}|_{\Sl}\big) \nonumber\\
\bleq & 4q\textstyle
\big(1+2\frac{c_2}{c_1}\big)^{2q-1}|\sn{i}-\sn{j}|_{\Sl}^{2q-1} \max_{k=0, \dots, n-1}\diffsn{k}. \label{absch_diffg_hoch_q}
\end{align}
Thus, by \eqref{dist_ij}, \eqref{absch_diffg_hoch_q}, and \eqref{nenner},
\begin{align}
|B_{i,j}|&
\bleq[\eqref{dist_ij}] \textstyle
K^q |\sn{i}-\sn{j}|_{\Sl}^{2q} \int_{\sn{j}}^{\sn{j+1}} \int_{\sn{i}}^{\sn{i+1}} \left|\frac{\diffgamma{\sn{i}}{\sn{j}}^{2q}-\diffgamma{s}{t}^{2q}}{\diffgamma{\sn{i}}{\sn{j}}^{2q}\diffgamma{s}{t}^{2q}}\right| \d s \d t \nonumber\\
\bleq[\eqref{absch_diffg_hoch_q}]& \textstyle
4q K^q \big(1+2\frac{c_2}{c_1}\big)^{2q-1}|\sn{i}-\sn{j}|_{\Sl}^{4q-1}\max_{k=0, \dots, n-1}\diffsn{k} \nonumber\\
&\textstyle
\times \int_{\sn{j}}^{\sn{j+1}} \int_{\sn{i}}^{\sn{i+1}}  \frac{1}{\diffgamma{\sn{i}}{\sn{j}}^{2q}\diffgamma{s}{t}^{2q}} \d s \d t \nonumber\\
\bleq[\eqref{nenner}]& \textstyle\big(\frac{2c_2 c_\gamma}{c_1}
\big)^{2q}c_\gamma^{2q} 4q K^q \big(1+2\frac{c_2}{c_1}
\big)^{2q-1}|\sn{i}-\sn{j}|_{\Sl}^{4q-1} \max_{k=0, \dots, n-1}\diffsn{k} \nonumber\\
&\textstyle \times \int_{\sn{j}}^{\sn{j+1}} \int_{\sn{i}}^{\sn{i+1}} \frac{1}{|\sn{i}-\sn{j}|_{\Sl}^{4q}} \d s \d t\nonumber\\
\bleq&\textstyle\big(\frac{2c_2 c_\gamma}{c_1}\big)^{2q}c_\gamma^{2q} 4q K^q\big(1+2\frac{c_2}{c_1}\big)^{2q-1} \big(\max_{k=0, \dots, n-1}\diffsn{k}\big)^3 \frac{1}{|\sn{i}-\sn{j}|_{\Sl}}\nonumber\\
\bleq[\eqref{prop_part}]& \textstyle\big(c_2\big)^3 
\big(\frac{2c_2 c_\gamma}{c_1}\big)^{2q}c_\gamma^{2q} 4q K^q
\big(1+2\frac{c_2}{c_1}\big)^{2q-1} \frac{1}{n^3} \frac{1}{|i-j|_n\underset{k=0, \dots, n-1}{\min}\diffsn{k}} 
\bleq[\eqref{prop_part}]\frac{C_B}{n^2}\frac{1}{|i-j|_n} \label{bij}
\end{align}
with $C_B:=\frac{(c_2)^3}{c_1} \big(\frac{2c_2 c_\gamma}{c_1}
\big)^{2q}c_\gamma^{2q} 4q K^q\big(1+2\frac{c_2}{c_1}\big)^{2q-1}$.

\textbf{Step 7:} The expression $\sum_{k=1}^n \frac{1}{k}-\ln(n)$ converges for $n \to \infty$ to the Euler-Mascheroni constant;
see \cite[(20.8) on p. 245]{forster_2011}. 
Thus, there exists a constant $c_l\in (0,\infty)$ such that 
$\left|\sum_{k=1}^n \frac{1}{k}-\ln(n)\right|\leq c_l$ for all 
$n \in \N$. Furthermore,  $\lim_{n \to \infty}\frac{\ln(n)}{n^\varepsilon}=0$ holds for all $\varepsilon>0$. 
Hence, for every $\varepsilon>0$ there exists a constant 
$\tilde{c}_\varepsilon \in (0,\infty)$ with $|\frac{\ln(n)}{n^\varepsilon}|\leq \tilde{c}_\varepsilon$ for all $n \in \N$. Combining these results leads to
\begin{align}
\label{sumAB}
\textstyle\sum_{i=0}^{n-1} &\textstyle\sum_{j, |i-j|_n>\Upsilon} 
\left(\left|A_{i,j}\right| + \left|B_{i,j}\right|\right)
\bleq[\eqref{aij},\eqref{bij}] \; \; \; \frac{2\max\{C_A,C_B\}}{n^2}\sum_{i=0}^{n-1} \sum_{j, |i-j|_n>\Upsilon} \frac{1}{|i-j|_n} \nonumber\\
&\bleq \textstyle
\frac{4\max\{C_A,C_B\}}{n^2}\sum_{i=0}^{n-1} \sum_{k=1}^n \frac{1}{k}\nonumber\\
&\beq\textstyle\frac{4\max\{C_A,C_B\}}{n}\left(\sum_{k=1}^n \frac{1}{k}-\ln(n)\right)+\frac{4\max\{C_A,C_B\}}{n^{1-\varepsilon}}\frac{\ln(n)}{n^\varepsilon}
\bleq\frac{C_{AB}}{n^{1-\varepsilon}}
\end{align}
with $C_{AB}:=8 \max\{c_l,\tilde{c}_\varepsilon\}\max\{C_A,C_B\}$.

\textbf{Step 8:} Since $\gamma \in C^{1,1}$, by \cite[Lemma 5.8]{smutny_2004}
\begin{align*}
\lambda_{n,j}-|\sn{j+1}-\sn{j}|=O\big(\left|\sn{j+1}-\sn{j}\right|^3\big) \text{ as }  n \to \infty\,\Foa j=0, \dots, n-1.
\end{align*}
Accordingly, there exists an $N \in \N$ and a constant $c_{N}\in (0,\infty)$ such that
\begin{align}
\label{Ordnung_lambda}
{\textstyle
\big|\lambda_{n,j}-|\sn{j+1}-\sn{j}|\big|\leq c_{N} \left|\sn{j+1}-\sn{j}
\right|^3} \leq  c_{N} \big(\max_{k=0, \dots, n-1}\diffsn{k}\big)^3
\end{align}
holds for all $n \geq N$. 
By \eqref{absch_bruch} and \eqref{lambda}  we obtain from 
\eqref{Ordnung_lambda}
\begin{align}
\label{cij}
|C_{i,j}|\bleq&
\textstyle
\frac{\distg{\sn{i}}{\sn{j}}^q}{\diffgamma{\sn{i}}{\sn{j}}^{2q}}\big|\left|\sn{i+1}-\sn{i}\right|\left|\sn{j+1}-\sn{j}\right|-\lambda_{n,i}\lambda_{n,j}\big|\nonumber\\
\bleq[\eqref{absch_bruch}]& \left(c_\gamma^2 K\right)^q \big[ \left|\sn{i+1}-\sn{i}\right|\big|\left|\sn{j+1}-\sn{j}\right|-\lambda_{n,j}\big|+\lambda_{n,j}\big|\left|\sn{i+1}-\sn{i}\right|-\lambda_{n,i}\big|\,\big]\nonumber\\
\bleq[\eqref{lambda}]& \left(c_\gamma^2 K\right)^q c_\lambda \big[ \left|\sn{i+1}-\sn{i}\right|\big|\left|\sn{j+1}-\sn{j}\right|-\lambda_{n,j}\big| \nonumber \\
&+\left|\sn{j+1}-\sn{j}\right|\big|\left|\sn{i+1}-\sn{i}\right|-\lambda_{n,i}\big|\big]\nonumber\\
\bleq[\eqref{Ordnung_lambda}]&\left(c_\gamma^2 K\right)^q c_\lambda c_{N} \big(\max_{k=0, \dots, n-1}\diffsn{k}\big)^3\big[\left|\sn{i+1}-\sn{i}\right|+\left|\sn{j+1}-\sn{j}\right|\big]\nonumber\\
\bleq& 2 \left(c_\gamma^2 K\right)^q c_\lambda c_{N} \big(\max_{k=0, \dots, n-1}\diffsn{k}\big)^4 
\bleq[\eqref{prop_part}]\frac{C_C}{n^4}
\end{align}
with $C_C:=2 \left(c_2\right)^4 \left(c_\gamma^2 K\right)^q c_\lambda c_{N}$ for all $n \geq N$. We then conclude
\begin{align}
\label{sumC}
\textstyle
\sum_{i=0}^{n-1} \sum_{j, |i-j|_n>\Upsilon} |C_{i,j}|\bleq[\eqref{cij}] \frac{C_C}{n^4}\sum_{i=0}^{n-1} \sum_{j=0}^{n-1} 1 = \frac{C_C}{n^2}.
\end{align}

\textbf{Step 9:} Inserting 
\eqref{Term1}, \eqref{Term2}, \eqref{sumAB} and \eqref{sumC} into
\eqref{zerl_energien} yields
\begin{align*}
\big| \TP_q\left(\gamma\right)-\tilde{\E}_q^n\left(\beta_n\right)\big|
\leq& \textstyle
\frac{C_{1}}{n}+ \frac{C_{2}}{n}+2^q\big( \frac{C_{AB}}{n^{1-\varepsilon}} + \frac{C_C}{n^2}\big)
\leq \textstyle
\frac{C_\varepsilon}{n^{1-\varepsilon}}
\end{align*}
with $C_\varepsilon:=4\max\{C_{1},C_{2},2^q C_{AB}, 2^q C_C\}$, which gives the desired result.
\end{proof}
\end{theorem}

\section{$\Gamma$-convergence to the continuous tangent-point energy}
\label{gamma_section}
In the present section,  we show 
that the continuous tangent-point energy $\TP_q$
is the 
$\Gamma$-limit of the discrete tangent-point energies
$\E_q^n$ as $n\to\infty$ (see Theorem \ref{gamma_convergence}). 
As a consequence, we deduce that limits of discrete almost minimizers 
are minimizers of the continuous tangent-point energy;
see Corollary \ref{cor:convergence-mini-tp}.

{\bf $\Gamma$-convergence.}\,
In order to prove Theorem \ref{gamma_convergence} we need to verify 
the liminf and limsup inequality, 
see \cite[Definition 1.5]{braides_2002}. Here, the liminf inequality
is verified in a rather straightforward manner (Theorem 
\ref{thm:liminfineq}),
 whereas the proof of the limsup inequality requires more
work; see Theorem \ref{limsup_proof} below.

\begin{theorem}[Liminf inequality]\label{thm:liminfineq}
Let $\gamma,\g_n \in C^1_\textnormal{a}\left(\Sl,\R^3\right)$ 
with $\gamma_n \stackrel{C^1}{\longrightarrow} \gamma$ as $n \to \infty$. Then
$
\TP_q(\gamma)\leq \liminf_{n \to \infty} \E_q^n(\gamma_n).
$
\begin{proof}
We may assume that $\liminf_{n \to \infty} \E_q^n(\gamma_n) < \infty$. Then there exists a subsequence $\left(\gamma_{n_k}\right)_{k \in\N}$ satisfying
$
\liminf_{n \to \infty} \E_q^n(\gamma_n)=\lim_{k \to \infty} \E_q^{n_k} \left(\gamma_{n_k}\right)<\infty.
$
By definition of $\E_q^{n_k}$ we deduce $\gamma_{n_k} \in \B_{n_k}$ 
for all $k \in \N$; see \eqref{eq:discrete-tan-point} in the 
introduction.
Denote the point-tangent pairs that are interpolated by $\gamma_{n_k}$  
as $\left(\left(\left[q_{{n_k},i},t_{{n_k},i}\right],\left[q_{{n_k},i+1},t_{{n_k},i+1}\right]\right)\right)_{i=0, \dots, {n_k}-1}$,
with $q_{{n_k},0}=q_{{n_k},{n_k}}$ and $t_{{n_k},0}=t_{{n_k},{n_k}}$
for each $k\in\N$. 
Furthermore, we denote by $a_{{n_k},0}, \dots, a_{{n_k},{n_k}}$ the arclength parameters satisfying $\gamma_{n_k}(a_{{n_k},i})=q_{{n_k},i}$ and $|a_{{n_k},i+1}-a_{{n_k},i}|=\lambda_{{n_k},i}$ for all $i=0, \dots, {n_k}-1$. Define for all $s,t \in \Sl$ with $s\neq t$ the function
\begin{align*}
f_{n_k}(s,t):=
\sum_{i=0}^{{n_k}-1} \sum_{j=0, j\neq i}^{{n_k}-1} \Big(2\frac{\dist\left(l(a_{{n_k},j}),\gamma_{n_k}(a_{{n_k},i})\right)}{|\gamma_{n_k}(a_{{n_k},i})-\gamma_{n_k}(a_{{n_k},j})|^2}\Big)^q \chi_{[a_{{n_k},i},a_{{n_k},i+1})\times[a_{{n_k},j},a_{{n_k},j+1})}(s,t),
\end{align*}
where $\chi_A$ denotes the characteristic function of a set
$A\subset\R/L\Z\times \R/L\Z.$
Easy calculations  show that 
\begin{align*}
\lim_{k \to \infty} f_{n_k}(s,t)=\textstyle
\left(\frac{2\dist(l(t),\gamma(s))}{|\gamma(s)-\gamma(t)|^2}\right)^q
\quad\Foa s\not= t.
\end{align*}
The functions $f_{n_k}$ are non-negative, and measurable since they 
are piecewise constant. Rewriting the discrete tangent-point energies 
as
$
\E^n_q(\g_{n_k})=\textstyle
\int_{\R/L\Z}\int_{\R/L\Z}f_{n_k}(s,t)\d s \d t,
$
allows us to apply Fatou's lemma to obtain the desired liminf inequality.
\end{proof}
\end{theorem}

An important first ingredient in the proof of the limsup inequality 
is the use of 
convolutions 
\begin{align}\label{eq:convolution}
\gamma_\varepsilon(x):=\left(\gamma \ast \eta_\varepsilon\right)(x)=
\textstyle
\int_\R \gamma(x-y)\eta_\varepsilon(y)\d y \quad\Fo x\in\Sl,
\end{align}
that approximate $\gamma$ in the $C^1$-norm. Here,
$\eta \in C^{\infty}\left(\R\right)$ is a non-negative
mollifier with $\supp \eta \subset [-1,1]$ and $\int_{\R} \eta(x)\d x =1$, 
and for any $\varepsilon>0$ we 
set $\eta_\varepsilon(x):=\frac{1}{\varepsilon}\eta\left(\frac{x}{\varepsilon} \right)$.

In general, the convolutions are not parametrized by arclength even
if $\g$ is, and they
do not need to have the same length as $\gamma$. Thus, we rescale 
the convolutions to have the same length as $\gamma$ and reparametrize
then according to arclength. The following theorem extends
\cite[Theorem 1.3]{blatt_2019b} to the case $\rho\ge\frac1s$.
A proof can be found in Appendix 
\ref{proof_convolution_convergence}. 

\begin{theorem}
\label{konv_faltung}
Let $s\in (0,1)$, $\rho \in [\frac1s,\infty)$, and $\gamma \in W_\textnormal{a}^{1+s,\rho} \left(\Sl,\R^3\right)$. For $\varepsilon>0$ denote by $\tilde{\gamma}_\varepsilon$ be the arclength 
parametrization of the rescaled convolutions
$\frac{\mathscr{L}\left(\gamma\right)}{\mathscr{L}\left(\gamma_\varepsilon\right)}\gamma_\varepsilon$ with $\tilde{\gamma}_\varepsilon(0)=\frac{\mathscr{L}\left(\gamma\right)}{\mathscr{L}\left(\gamma_\varepsilon\right)}\gamma_\varepsilon(0)$ where $\mathscr{L}\left(\gamma\right)$ is the length of $\gamma$ and $\mathscr{L}\left(\gamma_\varepsilon\right)$ is the length of $\gamma_\varepsilon$. Then $\tilde{\gamma}_\varepsilon \to\g$
in $W^{1+s,\rho}$ as $\varepsilon \to 0.$
\end{theorem}

The following abstract lemma provides sufficient conditions to
transfer the limsup inequality from approximating elements
to the limit element. This result applied to smooth
convolutions approximating a given $C^1$-curve $\g$ will be the second
ingredient in the proof of the limsup inequality, Theorem 
\ref{limsup_proof} below.

\begin{lemma}[Limsup inequality by approximation]
\label{limsup_approx}
Let $(X,d)$ be a metric space and $\mathcal{F}_n, \; \mathcal{F} : 
X \to [-\infty,\infty]$. If a sequence $\left(x^m\right)_{m \in \N}
\subset X$ satisfies
\begin{enumerate}
\item  $d\left(x,x^m\right) \to 0$ as $m \to \infty$ for an element $x \in X$,\label{konvergenz_folge}
\item $\limsup_{m \to \infty} \mathcal{F}\left(x^m\right)\leq \mathcal{F}(x)$,\label{stetigkeit_funktional}
\item for every $m \in \N$ there exists a sequence $\left(x^m_n\right)_{n \in \N}$ with $d\left(x^m,x^m_n\right) \to 0$ as $n \to \infty$ and $\limsup_{n \to \infty} \mathcal{F}_n\left(x^m_n\right)\leq \mathcal{F}\left(x^m\right)$, \label{limsup_bed}
\end{enumerate}
then there exists a sequence $\left(y_n\right)_{n \in \N}\subset X$ 
with 
\begin{align*}
d\left(x,y_n\right) \to 0 \text{ as } n \to \infty \quad \text{and} \quad \limsup_{n \to \infty} \mathcal{F}_n\left(y_n\right) \leq \mathcal{F}(x).
\end{align*}
\end{lemma}
We learnt this result from \cite[Lemma 1.0.4]{limberg_2019}; for the 
convenience of the reader we present its proof in Appendix \ref{app:C}.

The proof of the limsup inequality is inspired by Blatt's
improvement of Scholtes' $\Gamma$-convergence result for 
the M\"obius energy \cite[Theorem 4.8]{blatt_2019b}.

\begin{theorem}[Limsup inequality]
\label{limsup_proof}
For  
every $\gamma \in C^1_\textnormal{ia}
\left(\Sl,\R^3\right) $
there exists a sequence $\left(b_n \right)_{n \in \N} \subset C^1_\textnormal{ia}
\left(\Sl,\R^3\right)$
such that
\begin{align*}
b_n \stackrel{C^1}{\underset{n \to \infty}{\longrightarrow}} \gamma \quad \text{and} \quad \limsup_{n \to \infty} \E_q^{n} \left(b_n \right) \leq \TP_q(\gamma).
\end{align*}
\begin{proof}
If $\TP_q(\gamma)=\infty$, choose $b_n=\gamma$ for all $n \in \N$. Then $b_n\to \gamma$ in $C^1$ and the limsup inequality
follows trivially. From now on let $\TP_q(\gamma)<\infty$. 
Thus, we have \linebreak $\gamma \in 
W^{2-\frac{1}{q},q}\left(\Sl,\R^3\right)$ by 
\cite[Theorem 1.1]{blatt_2013b}. Moreover, Lemma
\ref{lem:bilipschitz} yields a $c_\gamma >0$ such that
\begin{align}
\label{bilipschitz}
c_\g|\gamma(s)-\gamma(t)|\geq  |s-t|_{\Sl}\quad\Foa t,s\in\R.
\end{align}
We now consider a sequence of suitably rescaled and
reparametrized convolutions of $\gamma$ and prove the limsup 
inequality for these convolutions. Applying Lemma \ref{limsup_approx} then yields the limsup inequality for $\gamma$.

\textbf{Step 1:} For $n \in \N$ define $\sn{i}:=\frac{iL}{n}$ for 
$i=0, \dots, n$. Then for all $i=0, \dots,n-1$ we have 
$\diffsn{i}=h_n=\tilde{h}_n=\frac{L}{n}$,
so that the $\mathcal{M}_n:=\{\sn{0},\ldots,\sn{n}\}$ form a
$(c_1-c_2)$-distributed sequence of partitions with $c_1=c_2=L$
for $n\ge 2$; see Definition \ref{def_part}. 

\textbf{Step 2:} For $k \in \N$ we set $\varepsilon_k:=\frac{1}{k}$. Let $\gamma_{\varepsilon_k}$ be the convolution as in \eqref{eq:convolution}
 and $\mathscr{L}\left(\gamma_{\varepsilon_k}\right)$ the length of 
 $\gamma_{\varepsilon_k}$. We then define $\tilde{\gamma}_k$ as the 
 arclength parametrization of the rescalings
 $L\gamma_{\varepsilon_k}/\mathscr{L}\left(\gamma_{\varepsilon_k}\right)$.
 Thus, $\tilde{\gamma}_k$ has the same length as $\gamma$ for every 
 $k \in \N$. Furthermore, $\tilde{\gamma}_k$ is on $[0,L)$ injective for 
 $k$ sufficiently large, which follows from the bilipschitz
 property \eqref{bilipschitz} of $\g$ together with the $C^1$-convergence
 of the convolutions $\g_{\varepsilon_k}\to \g$ as $k\to\infty$.
 By omitting finitely many indices we may assume that $\tilde{\g}_k
 \in C^\infty_\textnormal{ia}(\Sl,\R^3)$ for \emph{all} $k\in\N$.
For every $k\in\N$ there is by 
Lemma \ref{existenz_biarc_kurve_folge} 
some index $N_0(k)\in\N$ such 
that there exist proper 
$\tilde{\gamma}_k$-interpolating balanced biarc curves 
$\tilde{\beta}_n^k$ parametrized by arclength that interpolate the point-tangent pairs
$
\left(\left(\left[\tilde{\gamma}_k\left(\sn{i}\right),\tilde{\gamma}_k'\left(\sn{i}\right)\right],\left[\tilde{\gamma}_k\left(\sn{i+1}\right),\tilde{\gamma}_k'\left(\sn{i+1}\right)\right]\right)\right)_{i=0, \dots, n-1}$
for all $
n\ge N_0(k),
$
such that the matching points $m_{n,i}^k \in \Sigma_{++}^{n,i,k}$ satisfy
(see Definitions \ref{biarc_Kurve} and \ref{sigma++} (iv))
$
\left|\tilde{\gamma}_k\left(\sn{i}\right)-m_{n,i}^k\right|=\left|\tilde{\gamma}_k\left(\sn{i+1}\right)-m_{n,i}^k\right|
$ for all $ n\ge N_0(k),$ $
i=0,\ldots,n-1.$
Let $L_n^k:=\mathscr{L}(\tilde{\beta}_n^k)$ be the length of 
$\tilde{\beta}_n^k$, and notice that Theorem \ref{Laengenkonvergenz}
implies
\begin{equation}\label{eq:length-converges}
L_n^k\to L=\mathscr{L}(\tilde{\g}_k)\quad\textnormal{
for each $k\in\N$ as
$n\to\infty$.}
\end{equation}

\textbf{Step 3:} For $k \in \N$, let 
$\varphi_n^k$ be Smutny's  reparametrization \cite[Appendix A]{smutny_2004}
and  define $\tilde{B}_n^k:=\tilde{\beta}_n^k\circ\varphi_n^k,$ so that
Theorem \ref{C1_Konvergenz} implies
\begin{align}
\label{C1_B_n}
\|\tilde{\gamma}_k-\tilde{B}_n^k\|_{C^1} \to 0 \quad \textnormal{
for each $k\in\N$ as $ n \to \infty$.}
\end{align}
Now, define $B_n^k(s):=L \left(L_n^k\right)^{-1} 
\tilde{B}_n^k(s)$
for all $s \in \R$. Then $B_n^k$ has obviously length $L$.
However, $B_n^k$ is not parametrized by arclength. Nevertheless,
by means of \eqref{eq:length-converges} and \eqref{C1_B_n}
we find
$
\|B^k_n-\tilde{B}^k_n\|_{C^1}=\textstyle 
\big|\frac{L}{L^k_n}-1\big|\|\tilde{B}^k_n\|_{C^1}\to 0$ 
for each $k\in\N$ as $n\to\infty$. Consequently, by \eqref{C1_B_n},
one has $\|\tilde{\g}_k-B^k_n\|_{C^1}\to 0$ for each $k\in\N$ as 
$n\to\infty$, and therefore by means of Lemma \ref{c1_norm_bogenlaenge},
\begin{equation}\label{eq:beta-convergence}
\|\tilde{\gamma}_k - 
\beta_n^k\|_{C^1} \longrightarrow  0\quad\textnormal{for each $k\in\N$
as $n\to\infty$,}
\end{equation}
where $\beta_n^k$ is the reparametrization of $B_n^k$
by arclength.

\textbf{Step 4:} 
We now show that $\beta_n^k \in \mathcal{B}_n$ holds if $n$ is 
sufficiently large, such that the values $\E_q^n(\beta_n^k)$
 are finite by definition \eqref{eq:discrete-tan-point}
 in the introduction.
 We need to show that the length  $\lambda_{n,i}^k$ of the 
 $i$-th biarc of $\beta_n^k$ satisfies
 \eqref{eq:biarc-lengths}.
For that we apply Lemma \ref{Konv_max_bruch_lamda}
to the length
 $\tilde{\lambda}_{n,i}^k$ of the $i$-th biarc of 
 $\tilde{\beta}_n^k$. More precisely, we take the limit $n\to\infty$
 in the following inequality which holds 
 for each $k\in\N$, $n\ge N_0(k)$, $i=0,\ldots,n-1$,
 \begin{align*}
  -\max_{j=0, \dots, n-1} \left|
 {\textstyle \frac{\tilde{\lambda}_{n,j}^k}{({L}/{n})} -1} \right| 
 + 1\le {\textstyle\frac{\tilde{\lambda}_{n,i}^k}{({L}/{n})}} &\leq 
 \max_{j=0, \dots, n-1} \left|{\textstyle \frac{\tilde{\lambda}_{n,j}^k}{({L}/{n})}} -1 
 \right| + 1,
 \end{align*}
 to obtain 
 \begin{align}
1\longleftarrow \min_{i=0, \dots, n-1} {\textstyle
\frac{\tilde{\lambda}_{n,i}^k}{({L}/{n})}}\le \max_{i=0, \dots, n-1}
\textstyle
\frac{\tilde{\lambda}_{n,i}^k}{({L}/{n})} \longrightarrow 1
\quad\textnormal{as $n\to\infty$.}\label{new}
 \end{align}
Since the image of $\beta_n^k$ is just
the image of $\tilde{\beta}_n^k$ scaled by the factor 
$L \left(L_n^k\right)^{-1}$ we deduce $\lambda_{n,j}^k=L \left(L_n^k\right)^{-1} \tilde{\lambda}_{n,j}^k$. Combining this with \eqref{eq:length-converges}
we find for each $k\in\N$ and index $N_1(k)\ge N_0(k)$ such that
\begin{align}
\label{richtiger_Raum}
{\textstyle\frac{L}{2n}} \leq \min_{j=0, \dots,n-1}\lambda_{n,j}^k \leq \max_{j=0, \dots,n-1}\lambda_{n,j}^k \leq {\textstyle \frac{2L}{n}}\quad\Foa n\ge N_1(k),
\end{align}
which is \eqref{eq:biarc-lengths} for $\lambda_i:=\lambda_i^k$. 
Thus $\beta_n^k \in \mathcal{B}_n$ for all $n\ge N_1(k)$.

\textbf{Step 5:} 
The scaling property \eqref{scalings}
and the parameter invariance of the discrete  
 tangent-point energies 
 yields 
\begin{align*}
\E_q^n(\beta_n^k)
=\big(L (L_n^k)^{-1}\big)^{2-q} \tilde{\E}_q^n
(\tilde{\beta}_n^k ) \quad\textnormal{for all $k\in\N$ and $n\ge N_1(k)$,}
\end{align*}
so that we obtain by \eqref{eq:length-converges} and 
Theorem \ref{konv_ordnung_energien} applied to
$\g:=\tilde{\g}_k$ and $\beta_n:=\tilde{\beta}^k_n$
\begin{align}
|\TP_q(\tilde{\gamma}_k)-\E_q^n(\beta_n^k) |
&\leq |\TP_q(\tilde{\gamma}_k) - \tilde{\E}_q^n (\tilde{\beta}_n^k)|
+ |\tilde{\E}_q^n(\tilde{\beta}_n^k)|
\big|1-(L (L_n^k)^{-1} )^{2-q} \big|\longrightarrow 0\label{eq:energy-convergence}
\end{align}
for each $k\in\N$ as $n\to\infty$.

\textbf{Step 6:}
In this final step we check the assumptions of Lemma \ref{limsup_approx}.
The space $C_\textnormal{ia}^1\left(\Sl, \R^3 \right)$
 is a metric space with the metric induced by the $C^1$-norm.
 By the Morrey-Sobolev embedding 
 (see \cite[Theorem A.2]{knappmann-etal_2022} in the setting of periodic functions) there exists a constant $c_E >0$, such that
\begin{align*}
\| \tilde{\gamma}_k - \gamma\|_{C^1}\leq c_E\|\tilde{\gamma}_k - \gamma \|_{W^{2-\frac{1}{q},q}}.
\end{align*}
According to Theorem \ref{konv_faltung} applied to $\rho=q$ 
and $s=1-\tfrac{1}{q}$ for $q>2$ the right-hand side 
converges to $0$ as $k \to \infty$. Thus, $\tilde{\gamma}_k$ converges 
in the $C^1$-norm to $\gamma$, which verifies condition (i) in 
Lemma \ref{limsup_approx}. 
Furthermore, \cite[(4.2) Satz]{wings_2018} implies that $\TP_q$ is 
continuous on
$W_\textnormal{ia}^{2-\frac{1}{q},q}$ since 
$q > 2$. Thus, we obtain $\lim_{k \to \infty} \TP_q\left(\tilde{\gamma}_k\right) = \TP_q(\gamma)$,
which gives us condition (ii) of Lemma \ref{limsup_approx}. 
Combining \eqref{eq:beta-convergence} with \eqref{eq:energy-convergence}
verifies condition (iii) of Lemma \ref{limsup_approx}.
Hence,  Lemma \ref{limsup_approx}  yields  the limsup inequality 
for  $\gamma$. 
\end{proof}
\end{theorem}

\begin{remark}\label{rem:actually-biarcs}
For the proof of Theorem \ref{gamma_conv_rope} 
in Section \ref{sec:ropelength} it is important to note that
the actual recovery sequence for the limsup inequality in
the previous proof is
a subsequence of the (doubly subscripted)
arclength parametrized biarc curves
$\beta^k_n\in C^{1,1}_\textnormal{ia}(\Sl,\R^3)$ for $k\in\N$
and $n\ge N_1(k)$; see the choice of the abstract
recovery sequence $(y_n)_n$ towards the end of the
proof of Lemma \ref{limsup_approx} 
in Appendix
\ref{app:C}.
\end{remark}

{\it Proof of Theorem \ref{gamma_convergence}.}\,
According to \cite[Definition 1.5]{braides_2002} it suffices to
verify two fundamental inequalities. Indeed, the liminf inequality
is the content of Theorem \ref{thm:liminfineq}, 
whereas the limsup inequality
is established in Theorem \ref{limsup_proof}.\hfill $\Box$

{\bf Convergence of discrete almost minimizers.}\,
In this subsection, we prove the convergence of discrete almost 
minimizers of the discrete tangent-point energies in the metric space 
defined before. The following lemma of Scholtes 
\cite[Lemma A.1]{scholtes_2014b}
states the general result. For completeness we present its short
proof in Appendix \ref{app:C}.

\begin{lemma}[Convergence of minimizers]\label{konv_minimierer}
Let $(X,d)$ be a metric space and $\mathcal{F}_n, \mathcal{F} : 
X \to [-\infty,\infty]$, $Y\subset X$. Assume that $x_n \to x$ implies 
$\mathcal{F}(x)\leq \liminf_{n \to \infty} \mathcal{F}_n(x_n)$ and that 
for every $y \in Y$ there exists a sequence 
$\left(y_n\right)_{n \in \N}\subset X$ with $\limsup_{n \to \infty} 
\mathcal{F}_n\left(y_n\right)\leq \mathcal{F}(y)$.
Let $\left| \mathcal{F}_n(z_n) - \inf_X \mathcal{F}_n\right| \to 0$ and $z_n \to z \in X$ as $n$ tends to infinity. Then
$
\mathcal{F}(z) \leq \liminf_{n \to \infty} \inf_X \mathcal{F}_n \leq \inf_Y \mathcal{F}.
$
If, in addition, $z \in Y$, then
$
\mathcal{F}(z) = \min_Y \mathcal{F}$ and
$\lim_{n \to \infty} \mathcal{F}_n(z_n)=\mathcal{F}(z).
$
\end{lemma}

{\it Proof of Corollary \ref{cor:convergence-mini-tp}.}\,
The proof follows immediately from Lemma \ref{konv_minimierer} with 
the metric space
$
Y=X=C_\textnormal{ia}^1\left(\Sl, \R^3 \right)\cap\mathcal{K},
$
with metric induced by the $C^1$-norm. Notice that the knot class
$\mathcal{K}$
is stable under $C^1$-convergence; see, e.g., \cite{reiter_2005}.
Since $\TP_q(\gamma)<\infty$ holds, we obtain $\gamma \in W^{2-\frac{1}{q},q}\left(\Sl,\R^3\right)$ by \cite[Theorem 1.1]{blatt_2013b}.
\hfill $\Box $

\section{$\Gamma$-convergence to the Ropelength functional}
\label{sec:ropelength}
As a first step towards the proof of Theorem \ref{gamma_conv_rope}
we show that the continuous tangent-point energies 
$(\TP_k)^{\frac{1}{k}}$ $\Gamma$-converge  to the ropelength 
$\mathcal{R}$  on 
$C_\textnormal{ia}^{1,1}(\R/\Z,\R^3)$ equipped with
the $C^1$-norm
as $k\to\infty$. We follow the proof of 
\cite[Theorem 6.11]{gilsbach_2018}, where Gilsbach showed 
$\Gamma$-convergence of Integral Menger curvatures towards ropelength.

\begin{lemma}
\label{gamma_conv_cont_tpe_ropelength}
For any $\g\in C^{1,1}_\textnormal{ia}(\R/\Z,\R^3)$ one
has $(\TP_k)^{\frac1k}(\g)\to\rope(\g)$ as $k\to\infty$.
Moreover, 
$(\TP_k)^{\frac{1}{k}}  \stackrel{\Gamma}{\underset{k \to \infty}{\longrightarrow}} \rope$ on $(C_\textnormal{ia}^{1,1}(\Sc,\R^3),\|\cdot\|_{C^1} )$.
\end{lemma}
\begin{proof}
According to \cite[Theorem 1 (iii)]{schuricht-vdm_2003a}
one has\footnote{Be aware of the notation: In 
\cite{schuricht-vdm_2003a} the expression
$\mathcal{R}[\cdot]$ was used for 
thickness $\triangle[\cdot]$, whereas $\mathcal{K}[\cdot]$ in 
\cite{schuricht-vdm_2003a} corresponds to $\triangle[\cdot]^{-1}$.
}
$\rope(\gamma)<\infty$ for 
$\gamma \in C_\textnormal{ia}^{1,1}(\R/\Z,\R^3)$.
In addition, by \cite[Lemma 2]{schuricht-vdm_2003a}
\begin{align}
\|r^{-1}_\textnormal{tp}(\g(\cdot),\g(\cdot))\|_{L^\infty(\R/\Z\times
\R/\Z)}=
\sup_{s,t \in \Sc, s\neq t} \textstyle
r^{-1}_\textnormal{tp}(\g(s),\g(t)) = 
\frac{1}{\triangle[\gamma]}=\rope(\gamma).\label{linfty}
\end{align}
It is well-known (see, e.g., \cite[E3.4]{alt_2016}) that the mapping 
$$
k\mapsto \|r^{-1}_\textnormal{tp}(\g(\cdot),\g(\cdot))\|_{L^k(\R/\Z
\times\R/\Z)}=(\TP_k)^{\frac1k}
$$
is non-decreasing and satisfies  by means of \eqref{linfty}
$$
\lim_{k\to\infty}(\TP_k)^{\frac1k}(\g)
=\|r^{-1}_\textnormal{tp}(\g(\cdot),\g(\cdot))\|_{L^\infty(\R/\Z
\times\R/\Z)}\overset{\eqref{linfty}}{=}
\rope(\gamma).
$$
Furthermore, the continuous tangent-point energy is lower semi-continuous with respect to the $C^1$-norm, see 
\cite[Proof of Corollary 2.3]{strzelecki-etal_2013a} or
\cite[Lemma 1.41]{gilsbach_2018}. 
Then by \cite[Remark 1.40 (ii)]{braides_2002} the pointwise limit of 
$(\TP_k)^{\frac{1}{k}}$ is also the $\Gamma$-limit and we obtain
$
(\TP_k)^{\frac{1}{k}}
\stackrel{\Gamma}{\underset{}{\longrightarrow}} \rope
$ as $k\to\infty$.
\end{proof}

\begin{lemma}
\label{gamma_conv_discrete_cont}
 $(\E_q^n)^{\frac{1}{q}}  
 \stackrel{\Gamma}{\underset{n \to \infty}{\longrightarrow}} 
 (\TP_q)^{\frac{1}{q}}$ 
 on $(C_\textnormal{ia}^{1,1}(\Sc,\R^3),
 \|\cdot\|_{C^1} )$ for all $q > 2$.
\begin{proof}
By Theorem \ref{gamma_convergence} we have 
$\E_q^n \stackrel{\Gamma}{\underset{}{\longrightarrow}}\TP_q
$ on
$(C_\textnormal{ia}^1(\Sc, \R^3), \|\cdot \|_{C^1})$ for
any $q>2$ as $n\to\infty$.
However, in the proof of the limsup inequality in
Theorem \ref{limsup_proof} the recovery sequence is a sequence
consisting only of biarc curves that are in
$C_\textnormal{ia}^{1,1}\left(\Sc, \R^3 \right)$; see
Remark \ref{rem:actually-biarcs}. 
Therefore we also have 
$\E^n_q \stackrel{\Gamma}{\underset{}{\longrightarrow}}\TP_q
$ 
on the space
$(C_\textnormal{ia}^{1,1}(\Sc, \R^3 ), \|\cdot \|_{C^1})$
as $n\to\infty$.
Now apply Lemma \ref{gamma_konv_verkettung} in Appendix
\ref{app:C} for
$\mathcal{F}_n:=\E_q^n$, $\mathcal{F}:=\TP_q$ and
the continuous and non-decreasing
function $g : (0,\infty) \to \R, x \mapsto x^{\frac{1}{q}}$ 
to infer  
$
(\E_q^n)^{\frac{1}{q}} = g \circ \mathcal{F}_{n} 
\stackrel{\Gamma}{\underset{n \to \infty}{\longrightarrow}} 
g\circ \mathcal{F} = (\TP_q)^{\frac{1}{q}}.
$
\end{proof}
\end{lemma}

Next we compare two different discrete tangent-point energies.

\begin{lemma}
\label{Hoelder_discrete_energy}
Let $n,m,k \in \N$, $2\le k \leq m$ and 
$\g\in C^1(\R/\Z,\R^3)$
with length  $\mathscr{L}(\g)$. Then
$
\textstyle
\left(\E_k^n\right)^{\frac{1}{k}}(\g) \leq \left(\frac{4\mathscr{L}
(\g)^2n(n-1)}{n^2} \right)^{\frac{1}{k}-\frac{1}{m}}\left(\E_m^n\right)^{\frac{1}{m}}(\g).
$
\begin{proof}
We only have to consider the case that $\g\in\mathcal{B}_n$
since otherwise both sides of the inequality are infinite by definition
of the discrete energy $\E_k^n$;
see \eqref{eq:discrete-tan-point} in the introduction.
Denote by $\left(\left(\left[q_i,t_i\right],\left[q_{i+1},t_{i+1}\right] 
\right)\right)_{i=0, .., n-1}$ the point-tangent pairs that 
$\g$ interpolates. For $i\neq j$ define $x_{i,j}:=
\frac{2\;\dist(l(q_i),q_j)}{|q_i-q_j|^2}\geq 0$. Then we estimate
by means 
of the generalized mean inequality for finite sums,  
$(\frac1{\ell}\sum_{i=1}^\ell|a_i|^p)^{\frac1p}\le
(\frac1{\ell}\sum_{i=1}^\ell|a_i|^q)^{\frac1q}$ for $p\le q$
(here for $\ell:=n(n-1)$, $p:=k$, $q:=m$), and by
\eqref{eq:biarc-lengths}
\begin{align*}
\textstyle (\E_k^n)^{\frac{1}{k}}(\g)&
\textstyle
\beq \big(\sum_{i=0}^{n-1} \sum_{j=0,j\neq i}^{n-1} x_{i,j}^k 
\lambda_i \lambda_j \big)^{\frac{1}{k}}
\beq \big(\sum_{i=0}^{n-1} \sum_{j=0,j\neq i}^{n-1} 
\big[x_{i,j} (\lambda_i \lambda_j)^{\frac{1}{k}}\big]^k \big)^{\frac{1}{k}}\\
&\textstyle\bleq 
(n(n-1))^{\frac{1}{k}-\frac{1}{m}}\big(\sum_{i=0}^{n-1} \sum_{j=0,j\neq i}^{n-1} x_{i,j}^m \underbrace{\left(\lambda_i \lambda_j\right)^{\frac{m}{k}}}_{=\left(\lambda_i \lambda_j\right)\left(\lambda_i \lambda_j\right)^{\frac{m}{k}-1}} \big)^{\frac{1}{m}}\\
\bleq (n(n &-1))^{\frac{1}{k}-\frac{1}{m}} \max_{i,j=0,\dots,n-1,i\neq j} 
\textstyle
\underbrace{(\lambda_i \lambda_j)^{\frac{1}{k}-\frac{1}{m}}}_{\overset{\eqref{eq:biarc-lengths}}{\leq} 
\left(\frac{4\mathscr{L}(\g)^2}{n^2}\right)^{\frac{1}{k}-\frac{1}{m}}}
\underbrace{\textstyle\left(\sum_{i=0}^{n-1} \sum_{j=0,j\neq i}^{n-1} x_{i,j}^m \lambda_i \lambda_j\ \right)^{\frac{1}{m}}}_{=\left(\E_m^n\right)^{\frac{1}{m}}(\g)}\\
\end{align*}
\end{proof}
\end{lemma}
{\it Proof of Theorem \ref{gamma_conv_rope}.}\,
It suffices to prove the $\Gamma$-convergence for $L=1$, since
then the statement for general $L$ follows from
the scaling property and parametrization invariance of the energies 
involved. Indeed, assume the theorem was proven for $L=1$. Now take $L\neq 1$ and let $(\g_n)_{n\in \N}\subset
C^{1,1}_\textnormal{ia}(\Sl,\R^3)$ with $\g_n\to\g$ in $C^1$
as $n\to\infty$. Denote by $\tilde{\g}_n$ the arclength parametrization of $\frac{\g_n}{L}$. By Lemma \ref{c1_norm_bogenlaenge} this implies $\tilde{\g}_n \to \tilde{\g}$ in $C^1$ as $n \to \infty$, where $\tilde{\g}$ is the arclength parametrization of $\frac{\g}{L}$. Together with the fact 
that the ropelength functional is invariant under reparametrization and 
scaling, the liminf inequality for $L=1$ yields the liminf equality
for general $L$:
\begin{align*}
\rope(\g)=\rope(\tilde{\g})\leq 
\liminf_{n \to \infty}\left(\E_n^n\right)^{\frac{1}{n}}(\tilde{\g}_n)
=\liminf_{n \to \infty}\left(\E_n^n\right)^{\frac{1}{n}}{\textstyle
\big(\frac{\g_n}{L}\big)}
\overset{\eqref{scalings}}{=}
\liminf_{n \to \infty}L^{\frac{n-2}{n}}\left(\E_n^n\right)^{\frac{1}{n}}\left(\g_n\right).
\end{align*}
For the limsup inequality let $\g \in C^{1,1}_\textnormal{ia}(\Sl,\R^3)$. Then $\tilde{\g}(x):=\frac{\g(Lx)}{L}$ is the arclength parametrization of $\g$ scaled to unit length. Hence, there exists a  recovery sequence $(\tilde{\g}_n)_{n\in \N}\subset
C^{1,1}_\textnormal{ia}(\Sc,\R^3)$ such that
\begin{align}
\tilde{\g}_n \stackrel{C^1}{\to} \tilde{\g} \text{ as } n \to \infty \quad \text{and} \quad \limsup_{n \to \infty}\left(\E_n^n\right)^{\frac{1}{n}}(\tilde{\g}_n)\leq \rope(\tilde{\g}).
\label{recovery_scaled}
\end{align}
Define the reparametrization
$\varphi: [0,L] \to [0,1], x \mapsto \frac{x}{L}$ and set $\g_n(x):=
L
\tilde{\g_n}(\varphi(x))$ and $\hat{\g}(x):=
L\tilde{\g}(\varphi(x))$. Note that $\g_n$ is parametrized by 
arclength and that $\hat{\g}=\g$ holds. Then $\g_n \to \hat{\g}=\g$ in $C^1$ for $n \to \infty$ by \eqref{recovery_scaled}. Again, by the scaling property of the energies and the invariance under reparametrization we deduce with \eqref{recovery_scaled}
\begin{align*}
\limsup_{n \to \infty}
L^{\frac{n-2}{n}}\left(\E_n^n\right)^{\frac{1}{n}}(\g_n)
\overset{\eqref{scalings}}{=}
\limsup_{n \to \infty}
\left(\E_n^n\right)^{\frac{1}{n}}(\tilde{\g}_n)\leq 
\rope(\tilde{\g})=\rope(\g).
\end{align*}
So, it remains to prove the statement of Theorem 
\ref{gamma_conv_rope}
for $L=1$, and for that we take
a general sequence $(\g_n)_{n\in\N}\subset
C^{1,1}_\textnormal{ia}(\Sc,\R^3)$ with $\g_n\to\g$ in $C^1$
as $n\to\infty$. 

By Lemma \ref{gamma_conv_discrete_cont} for $q:=k$ 
\begin{align}
\label{ineq2}
(\TP_k)^{\frac{1}{k}}(\gamma)\leq \liminf_{n \to \infty}
(\E_k^n)^{\frac{1}{k}}(\gamma_n)=
\lim_{n\to\infty}\inf_{n \geq k}
(\E_k^n)^{\frac{1}{k}}(\gamma_n).
\end{align}
For $k \leq n$ we apply Lemma \ref{Hoelder_discrete_energy} to 
$\g:=\g_n$ and $m:=n$ to find
\begin{align*}
(\E_k^n)^{\frac{1}{k}}(\gamma_n) \leq \textstyle
\big(\frac{4\mathscr{L}(\gamma_n)^2 n(n-1)}{n^2} 
\big)^{\frac{1}{k}-\frac{1}{n}}(\E_n^n)^{\frac{1}{n}}(\gamma_n).
\end{align*}
Together with \eqref{ineq2} and $\mathscr{L}(\g_n)=1$ for all $n\in\N$,
this yields
\begin{align}
\label{ineq3}
 (\TP_k)^{\frac{1}{k}}(\gamma) \leq \lim_{n \to \infty}
\inf_{n \geq k}\textstyle\big(\frac{4n(n-1)}{n^2} \big)^{\frac{1}{k}-
\frac{1}{n}}(\E_n^n)^{\frac{1}{n}}(\gamma_n).
\end{align}
Now we have
\begin{align*}
\lim_{n \to \infty} {\textstyle
\big(\frac{4n(n-1)}{n^2} \big)^{\frac{1}{k}-\frac{1}{n}}}=
\lim_{n \to \infty} \exp\textstyle
\big( \left(\frac{1}{k}-\frac{1}{n} \big)\log\left(\frac{4n(n-1)}{n^2} 
\right)\right)=\exp\left(\frac{1}{k}\log(4) \right)=4^{\frac{1}{k}}.
\end{align*}
Combining this with the pointwise convergence in Lemma
\ref{gamma_conv_cont_tpe_ropelength} 
and \eqref{ineq3} we arrive at the desired liminf inequality:
\begin{align*}
\rope(\gamma)&\beq \lim_{k \to \infty} (\TP_k)^{\frac{1}{k}}(\gamma)
\bleq[\eqref{ineq3}] \lim_{k \to \infty} \lim_{n \to \infty}
\inf_{n \geq k}\textstyle
\big(\frac{4n(n-1)}{n^2} \big)^{\frac{1}{k}-\frac{1}{n}}
(\E_n^n)^{\frac{1}{n}}(\gamma_n)\\
&\beq \lim_{k \to \infty} 4^{\frac{1}{k}} \liminf_{n \to \infty} \left(\E_n^n\right)^{\frac{1}{n}}(\gamma_n)
\beq \liminf_{n \to \infty} \left(\E_n^n\right)^{\frac{1}{n}}(\gamma_n).
\end{align*}

To verify the limsup inequality let 
$\gamma \in  C_\textnormal{ia}^{1,1}\left(\Sc,\R^3\right)$, 
and for $n \in \N$ set $s_{n,i}:=\frac{i}{n}$ for $i=0,\dots,n$. 
Then we have $|s_{n,i+1}-s_{n,i}|=\frac{1}{n}$ for all $i=0,\dots,n-1$,
and therefore a sequence of $(c_1-c_2)$-distributed
partitions with $c_1=c_2=1$; see Defintition \ref{def_part}.
 Now we follow the proof of Theorem \ref{limsup_proof}. However, 
 since 
 $\gamma$ is now 
 a $C^{1,1}$-curve, we do not have to work with convolutions, but can 
 follow the proof for $\gamma$ directly. By Lemma 
 \ref{existenz_biarc_kurve_folge} there exists for $n$ sufficiently 
 large a $\gamma$-interpolating, proper and 
 balanced biarc curve $\tilde{\beta}_n$ interpolating 
 the point-tangent pairs
$
(([\gamma(\sn{i}),\gamma'(\sn{i})],[\gamma(\sn{i+1}),\gamma'(\sn{i+1})]))_{i=0, \dots, n-1}.
$
Then we obtain by Theorem \ref{Laengenkonvergenz} that
$\mathscr{L}(\tilde{\beta}_n)\to \mathscr{L}(\g)=1$ as $n\to\infty$.
Let $\varphi_n$ be the reparametrization function from 
\cite[Appendix A]{smutny_2004} and set $\tilde{B}_n:=\tilde{\beta}_n
\circ\varphi_n$. Then by Theorem \ref{C1_Konvergenz}, 
we have $\tilde{B}_n \to \gamma$ in $C^1$ for $n \to \infty$. 
Setting $B_n:=\mathscr{L}(\tilde{\beta}_n)^{-1} \tilde{B}_n$ 
we obtain  as in the proof of Theorem \ref{limsup_proof} that 
$B_n \to \gamma$ in $C^1$ for $n \to \infty$. 
Let $\beta_n$ be the arclength parametrization of $B_n$. 
By Lemma \ref{c1_norm_bogenlaenge} we finally arrive at 
$\beta_n \to \gamma$ in $C^1$ for $n \to \infty$. The biarc-curves
$\beta_n$ are only  reparametrized versions 
of $\tilde{\beta}_n$ rescaled by the factor 
$\mathscr{L}(\tilde{\beta}_n)^{-1}$
so that we can show exactly as 
in the proof of Theorem \ref{limsup_proof}  that 
$\beta_n \in \mathcal{B}_n$ for $n$ sufficiently large. Moreover,
due to the $C^1$-convergence towards $\g$, the
$\beta_n$ are  also injective for $n$ large enough.
Since $\beta_n$ is scaled to unit length and parametrized by arclength, 
we have $\beta_n \in C_\textnormal{ia}^{1,1}(\Sc,\R^3)$ for 
$n$ sufficiently large. Set $L_n:=\mathscr{L}(\tilde{\beta}_n)$.
By the scaling property of the discrete 
tangent-point 
energy \eqref{scalings}
and its parameter invariance  
 we have
\begin{align}
\label{scaling}
(\E_k^n)^{\frac{1}{k}}(\beta_n)=
L_n^{1-\frac{2}{k}}
(\E_k^n)^{\frac{1}{k}}(\tilde{\beta}_n)\quad\Foa k > 2.
\end{align}
Abbreviating
$
x_{i,j}:=\frac{2\dist(\gamma(s_{n,i})+\R\gamma'(s_{n,i}),\gamma(s_{n,j}))}{\left| \gamma(s_{n,i})-\gamma(s_{n,j})\right|^2}
$
for $i,j=0,\dots,n-1$ with $i\neq j$ we can write and estimate
for sufficiently large $n\in\N$
\begin{align}
\left(\E_k^n\right)^{\frac{1}{k}}(\beta_n)&\beq[\eqref{scaling}]L_n^{1-\frac{2}{k}}\left(\E_k^n\right)^{\frac{1}{k}}(\tilde{\beta}_n)
\beq \textstyle
L_n^{1-\frac{2}{k}} \big(\sum_{i=0}^{n-1} \sum_{j=0,j\neq i}^{n-1} x_{i,j}^k \lambda_i \lambda_j \big)^{\frac{1}{k}}\notag\\
&\textstyle
\beq L_n^{1-\frac{2}{k}} \big(\frac{1}{n(n-1)} \sum_{i=0}^{n-1} 
\sum_{j=0,j\neq i}^{n-1} n(n-1) x_{i,j}^k 
\lambda_i \lambda_j
\big)^{\frac{1}{k}}\notag\\
&\textstyle
\bleq L_n^{1-\frac{2}{k}} \big(\frac{4L_n^2 n (n-1)}{n^2} 
\big)^{\frac{1}{k}}\big(\frac{1}{n(n-1)} \sum_{i=0}^{n-1} 
\sum_{j=0,j\neq i}^{n-1}  x_{i,j}^k \big)^{\frac{1}{k}}.\label{4}
\end{align}
Here, we used \eqref{eq:biarc-lengths} for 
$\tilde{\beta}_n\in\mathcal{B}_n$, which can be verified for
$n$ sufficiently large exactly as 
for the $\tilde{\beta}_n^k$
in Step 4 of the proof of Theorem \ref{limsup_proof}; see in particular
\eqref{new}.
Now, observe that by
\cite[E3.4]{alt_2016} applied to the
discrete measure $\mu:=\sum_{i=0}^{n-1} \sum_{j=0,j\neq i}^{n-1}
\delta_{(\sn{i},\sn{j})}$ with $\mu((\Sl)^2)=n(n-1)$  we have
\begin{align}\label{3}
\lim_{k \to \infty}{\textstyle \big(
\frac{1}{n(n-1)} \sum_{i=0}^{n-1} \sum_{j=0,j\neq i}^{n-1}  x_{i,j}^k 
\big)^{\frac{1}{k}}}=\max_{i,j=0,\dots,n-1,i\neq j}
x_{i,j}.
\end{align}
With 
$
\lim_{k \to \infty} L_n^{1-\frac{2}{k}} \big(\frac{4L_n^2 n (n-1)}{n^2} 
\big)^{\frac{1}{k}}=L_n.
$
we obtain by means of \eqref{4}, \eqref{3}, and 
\cite[Lemma 2.5]{schuricht-vdm_2003a} for $n$ sufficiently large
\begin{align}
\begin{split}
\label{limsup_ineq1}
\limsup_{k \to \infty} \left(\E_k^n\right)^{\frac{1}{k}}(\beta_n) &
\overset{\eqref{4},\eqref{3}}{\leq} 
L_n \max_{i,j=0,\dots,n-1,i\neq j} x_{i,j}
\leq L_n \sup_{s,t \in \Sc, s\neq t} \textstyle
\frac{2\dist(\gamma(s)+\R\gamma'(s),\gamma(t))}{\left| \gamma(s)-\gamma(t)\right|^2}\\
&=L_n  \sup_{s,t \in \Sc, s\neq t}\textstyle
\frac{1}{r_\textnormal{tp}
(\gamma(s),\gamma(t))}
=L_n \rope(\gamma).
\end{split}
\end{align}
Now let $k \geq n$. By virtue of Lemma \ref{Hoelder_discrete_energy} 
applied to $\beta_n$, $m:=k$ and replacing the index $k$ in
that lemma by $n$ here,  we have
$
\left(\E_n^n\right)^{\frac{1}{n}}(\beta_n)\leq 
\textstyle(\frac{4L_n n (n-1)}{n^2} )^{\frac{1}{n}-\frac{1}{k}} 
(\E_k^n)^{\frac{1}{k}}(\beta_n),
$
which leads to
\begin{align}
\label{limsup_ineq3}
\left(\E_n^n\right)^{\frac{1}{n}}(\beta_n) &\leq 
\limsup_{k \to \infty}\textstyle\big(\frac{4L_n n (n-1)}{n^2} 
\big)^{\frac{1}{n}-\frac{1}{k}} \left(\E_k^n\right)^{\frac{1}{k}}(\beta_n)\bleq[\eqref{limsup_ineq1}]\textstyle
\big(\frac{4L_n n (n-1)}{n^2} \big)^{\frac{1}{n}} L_n \rope(\gamma)
\end{align}
for $n$ sufficiently large.
Finally, taking the limsup yields the desired limsup inequality
\begin{align*}
\limsup_{n \to \infty} \left(\E_n^n\right)^{\frac{1}{n}}(\beta_n)
\bleq[\eqref{limsup_ineq3}] \limsup_{n \to \infty} 
\underbrace{\big(\textstyle\frac{4L_n n (n-1)}{n^2} 
\big)^{\frac{1}{n}}}_{\to 1} \underbrace{L_n}_{\to \mathscr{L}(\g)=1} \rope(\gamma)=\rope(\gamma)
\end{align*}
\hfill $\Box$

{\it Proof of Corollary \ref{conv_mini_ropelength}.}\,
Apply Lemma \ref{konv_minimierer} to 
the metric space
$
Y=X=C_\textnormal{ia}^{1,1}\left(\Sc, \R^3 \right)
\cap\mathcal{K}
$
with metric induced by the $C^1$-norm.
Notice as in the proof of Corollary \ref{cor:convergence-mini-tp}
that according to \cite{reiter_2005} the knot class
$\mathcal{K}$
is stable under $C^1$-convergence.
Since $\rope(\gamma)<\infty$ holds, we obtain by
\cite[Lemma 2]{gonzalez-etal_2002} that
$\gamma \in C_\textnormal{ia}^{1,1}\left(\Sc, \R^3 \right)$.

\appendix
\section{Convergence of convolutions in $W^{2-\frac{1}{q},q}(\Sl,\R^3)$}
\label{proof_convolution_convergence}
For fixed $L>0$, $s\in (0,1)$ and $\rho\in [1,\infty)$
define the seminorm $[f]_{s,\rho}$
of an $L$-periodic locally $\rho$-integrable
function $f:\R\to\R^n$ as
\begin{equation}\label{semi-norm}
[f]_{s,\rho}:=\textstyle
\int_{\Sl}\int_{\Sl}\frac{|f(x)-f(y)|^\rho}{|x-y|_{\R/L\Z}^{1+s\rho}}
\d x \d y,
\end{equation}
where $|x-y|_{\R/L\Z}$ denotes the periodic distance on $\R$ defined
in \eqref{eq:periodic-norm}. Then the \emph{periodic fractional\footnote{Also known as periodic Sobolev-Slobodecki\v{\i} space.} Sobolev space}
$W^{1+s,\rho}(\R/L\Z,\R^n)$ consists of those Sobolev functions 
$f\in W^{1,\rho}(\R/L\Z,\R^n)$ whose weak derivatives $f'$ have finite
 seminorm $[f']_{s,\rho}$.
The norm on $W^{1+s,\rho}(\R/L\Z,\R^n)$ 
is given by $(\|f\|_{W^{1,\rho}}+
[f']_{s,\rho})^{\frac1\rho}$

{\it Proof of Theorem \ref{konv_faltung}.}\,
The case $\rho=\frac1s$ is treated in 
\cite[Theorem 1.3]{blatt_2019b},
so we may assume from now on that $\rho>\frac1s $.\\
\textbf{Step 1:}
According to the Morrey-Sobolev embedding
\cite[Theorem A.2]{knappmann-etal_2022} we have
$\g\in C^1(\R/L\Z,\R^3)$
which implies that $\g'$ is of \emph{vanishing mean oscillation},
in short $\g'\in\VMO(\R/L\Z,\R^3)$, 
that is
$
\textstyle
\lim_{r \to 0} (\sup_{x \in \Sl} \frac{1}{2r}(\int_{B_r(x)} |\gamma'(y) - {\gamma'}_{x,r} |\d y ) )=0,
$
where
$
{\gamma'}_{x,r}:= \frac{1}{2r}\int_{B_r(x)} \gamma'(z)\d z
$
denotes the integral mean. Indeed, 
$\gamma'$ is uniformly continuous so that for every $\varepsilon>0$
there exists a $\delta=\delta(\varepsilon)>0$ such that 
$\left| \gamma'(x)-\gamma'(y)\right|<\frac{\varepsilon}{2}$ for all $x \in \Sl$ and $y \in B_{\delta}(x)$. Let $0<r<\delta$ and $x \in \Sl$. Then
\begin{align*}
\textstyle
\frac{1}{2r} \int_{B_r(x)} 
|\gamma'(y) & - {\gamma'}_{x,r} |\d y 
\leq \sup_{y \in B_r(x)} |\gamma'(y) - \gamma'(x)| + 
|\gamma'(x) - {\gamma'}_{x,r}|\\
& \le 
\textstyle
\frac{\varepsilon}{2} + \frac{1}{2r}\int_{B_r(x)} 
|\gamma'(x)-\gamma'(z) |\d z
< \varepsilon\quad\Foa x\in\R/L\Z,
\end{align*}
thus
$
\sup_{x \in \Sl} \frac{1}{2r}(\int_{B_r(x)} |\gamma'(y) - 
{\gamma'}_{x,r} |\d y ) <\varepsilon,
$
which implies that $\g'\in\VMO(\Sl,\R^3)$ 
since $\varepsilon>0$ was arbitrary. 

\textbf{Step 2:} 
For the lengths $L_\varepsilon:=\mathscr{L}(\gamma_\varepsilon)$ 
and $L:=\mathscr{L}(\g)$ we  estimate
\begin{align}
\label{length_faltung}
\textstyle
|L_\varepsilon-L|
&\leq \textstyle
\int_0^L \big| |\gamma'_\varepsilon(x) | 
- |\gamma'(x) |\big| \d x 
\leq \| \;\left|\gamma'_\varepsilon \right| - \left|\gamma' \right|\;\|_{C^0} L\stackrel{}{\underset{\varepsilon \to 0}{\longrightarrow}}0,
\end{align}
since $\g'\in\VMO(\Sl,\R^3)$ allows us to apply 
\cite[Theorem 1.1]{blatt_2019b} to deduce that 
$\left|\gamma'_\varepsilon \right|$ converges uniformly to 
$|\gamma'|=1$ as $\varepsilon$ tends to $0$. Therefore, there is an
$\varepsilon_0>0$ such that
\begin{align}
\label{bounds_derivative}
\textstyle \frac{1}{2}\leq \big|\frac{L}{L_\varepsilon}\g_\varepsilon' (x)\big|\leq 2\quad\Foa x\in\Sl,\,\varepsilon \in (0,\varepsilon_0].
\end{align}

\textbf{Step 3:} 
Since the convolutions $\g_\varepsilon$ converge to $\g$ in $C^1$
we obtain by means of \eqref{length_faltung} that also the 
rescalings $L\g_\varepsilon/L_\varepsilon$ converge towards $\g$
in $C^1$. According to Lemma \ref{c1_norm_bogenlaenge} we obtain
\begin{equation}\label{double-plus}
\|\tilde{\g}_\varepsilon-\g\|_{C^1}\longrightarrow 0\quad\textnormal{as $\varepsilon\to 0$.}
\end{equation}

\textbf{Step 4:} It remains to show that
$
\left[\tilde{\gamma}'_\varepsilon - \gamma' \right]_{s,\rho} \to 0
$
holds as $\varepsilon \to 0$, since then together with 
\eqref{double-plus} 
we have established $\|\tilde{\g}_\varepsilon-\g\|_{W^{1+s ,\rho}}\to 0$ as $\varepsilon\to 0$. 
Abbreviating the integrand of the seminorm by 
$
\textstyle
I_\varepsilon(x,y):= \frac{\left|\left(\tilde{\gamma}'_\varepsilon(x)- \gamma'(x)\right)-\left(\tilde{\gamma}'_\varepsilon(y)-\gamma'(y) \right) \right|^\rho}{|x-y|_{\Sl}^{1+s\rho}}
$
we
want to apply Vitali's theorem (see, e.g., \cite[3.23]{alt_2016})
to prove 
$\|I_\varepsilon\|_{L^1} \to 0$ 
as $\varepsilon \to 0$. 
Since we have a compact domain it suffices to show that the 
sequence $\left(I_\varepsilon\right)_{\varepsilon>0}$ is uniformly 
integrable and converges pointwise to $0$ a.e. on
$\Sl\times\Sl$.
The pointwise convergence $I_\varepsilon(x,y)\to 0$ (even for all $x\not=y$)
follows from the $C^1$-convergence \eqref{double-plus}.

So, we need to show the uniform integrability. 
In the obvious inequality
\begin{align}\label{obvious}
\textstyle
I_\varepsilon(x,y)\leq 2^{\rho-1} \left[\frac{\left|\tilde{\gamma}'_\varepsilon(x)-\tilde{\gamma}'_\varepsilon(y) \right|^\rho}{|x-y|_{\Sl}^{1+s\rho}} + \frac{\left|\gamma'(x)-\gamma'(y)\right|^\rho}{|x-y|_{\Sl}^{1+s\rho}} \right]
\end{align}
we estimate both summands on the right-hand side separately.

This is easy for the second summand. Fix
$\tilde{\varepsilon}>0$. Since $\gamma$ is in $W^{1+s,\rho}(\Sl,\R^3)$, 
we find
a $\delta_1=\delta_1(\tilde{\varepsilon})>0$ such that for every 
measurable subset $E \subset \left(\Sl\right)^2$ with $|E|<\delta_1$ we obtain
\begin{align}
\label{second_term_integrable}
\textstyle
\int \int_E \frac{\left|\gamma'(x)-\gamma'(y) \right|^\rho}{|x-y|_{\Sl}^{1+s\rho}}\d x\d y < \frac{\tilde{\varepsilon}}{2^\rho}.
\end{align}
Regarding the first summand in \eqref{obvious} 
we consider the arclength function
$
\textstyle
s_\varepsilon(x):=\int_0^x |\frac{L}{L_\varepsilon}\gamma'_\varepsilon(z) |\d z
$ such that $s_\varepsilon(0)=0$.
By \eqref{bounds_derivative}
the derivative
$
s'_\varepsilon(z)=|\frac{L}{L_\varepsilon}\gamma'_\varepsilon(z) |
$
is uniformly bounded away from $0$ for all $\varepsilon\in
(0,\varepsilon_0]$.
As a consequence, $s_\varepsilon$ is 
for $\varepsilon\in (0,\varepsilon_0]$ invertible. 
Let $\tilde{s}_\varepsilon$ denote the inverse function of 
$s_\varepsilon$. 
As a next step, we will show $|s_\varepsilon(x)-s_\varepsilon(y)|_{\Sl}\geq \frac{1}{2}|x-y|_{\Sl}$ for $x,y \in \Sl$ and $\varepsilon\in (0,\varepsilon_0]$. Let $0\leq x < y < L$, so that by monotonicity $0\leq s_\varepsilon(x)\leq s_\varepsilon(y)<L$. First, assume that $|s_\varepsilon(x)-s_\varepsilon(y)|_{\Sl}=|s_\varepsilon(x)-s_\varepsilon(y)|=s_\varepsilon(y)-s_\varepsilon(x)$.  Then we estimate by means of \eqref{bounds_derivative}
\begin{align*}
\textstyle
|s_\varepsilon(x)-s_\varepsilon(y)|_{\Sl}
=\int_x^y \big|\frac{L}{L_\varepsilon}\gamma'_\varepsilon(z) \big|\d z
\overset{\eqref{bounds_derivative}}{\geq} \frac{1}{2}(y-x)\geq \frac{1}{2}|x-y|_{\Sl}.
\end{align*}
If $|s_\varepsilon(x)-s_\varepsilon(y)|_{\Sl}=L-(s_\varepsilon(y)-s_\varepsilon(x))$, then again by \eqref{bounds_derivative}
\begin{align*}
|s_\varepsilon(x)-s_\varepsilon(y)|_{\Sl}&=\textstyle
L-\int_x^y \big|\frac{L}{L_\varepsilon}\gamma'_\varepsilon(z) \big|\d z=\int_0^x \big|\frac{L}{L_\varepsilon}\gamma'_\varepsilon(z) \big|\d z + \int_y^L \big|\frac{L}{L_\varepsilon}\gamma'_\varepsilon(z) \big|\d z\\
&\overset{\eqref{bounds_derivative}}{\geq}\textstyle 
\frac{1}{2}(L-(y-x))\geq \frac{1}{2}|x-y|_{\Sl}.
\end{align*}
In particular, this yields for the inverse function
\begin{align}
\label{inverse_lipschitz}
|\tilde{s}_\varepsilon(x)-\tilde{s}_\varepsilon(y)|_{\Sl}\leq 2 |x-y|_{\Sl}\quad\Foa x, y \in \Sl,\,\varepsilon \in (0,\varepsilon_0].
\end{align}
Due to 
\eqref{length_faltung}
there exists a 
constant $c>0$ such that
\begin{align}
\label{constante_bruch_L}
\textstyle
2^{(2+s)\rho}\left|\frac{L}{L_\varepsilon}\right|^\rho\left[1+2^\rho\left|\frac{L}{L_\varepsilon}\right|^\rho \right] \leq c \quad
\Foa \varepsilon>0.
\end{align}
Now we estimate pointwise for $x \neq y$ with Jensen's inequality
and by \eqref{inverse_lipschitz} 
\begin{align*}
J_\varepsilon(x,y)&\textstyle
:=\frac{\left|\tilde{\gamma}'_\varepsilon(x)-
\tilde{\gamma}'_\varepsilon(y)\right|^\rho}{|x-y|_{\Sl}^{1+s\rho}}
=\frac{\left|\frac{L}{L_\varepsilon}
\g_\varepsilon'(\tilde{s}_\varepsilon(x))\tilde{s}_\varepsilon'(x)
-\frac{L}{L_\varepsilon}\g_\varepsilon'(\tilde{s}_\varepsilon(y))
\tilde{s}_\varepsilon'(y)\right|^\rho}{|\tilde{s}_\varepsilon(x)-\tilde{s}_\varepsilon(y)|_{\Sl}^{1+s\rho}}
\frac{|\tilde{s}_\varepsilon(x)-\tilde{s}_\varepsilon(y)|_{\Sl}^{1+s\rho}}{|x-y|_{\Sl}^{1+s\rho}}\\
&\overset{\eqref{inverse_lipschitz}}{ \leq} \textstyle
2^{(1+s)\rho} \big|\frac{L}{L_\varepsilon}\big|^\rho 
\Big[\frac{|\tilde{s}_\varepsilon'(x)|^\rho
|\g_\varepsilon'(\tilde{s}_\varepsilon(x))
-\g_\varepsilon'(\tilde{s}_\varepsilon(y))|^\rho}{|\tilde{s}_\varepsilon(x)
-\tilde{s}_\varepsilon(y)|_{\Sl}^{1+s\rho}}
+\frac{|\g_\varepsilon'(\tilde{s}_\varepsilon(y))|^\rho
|\tilde{s}_\varepsilon'(x)-\tilde{s}_\varepsilon'(y)|^\rho}{|\tilde{s}_\varepsilon(x)-\tilde{s}_\varepsilon(y)|_{\Sl}^{1+s\rho}} \Big].
\end{align*}
Together with $|\g_\varepsilon'|\le 1$, $|s_\varepsilon'(x)|=
|\frac{L}{L_\varepsilon}\g_\varepsilon'(x)|\in 
\left[\tfrac{1}{2},2\right]$ due to \eqref{bounds_derivative}, 
$|\tilde{s}_\varepsilon'(x)|=
{|\frac{L}{L_\varepsilon}\g_\varepsilon'(\tilde{s}_\varepsilon(x))|^{-1}} 
\in \left[\tfrac{1}{2},2\right]$ for $\varepsilon \in (0,\varepsilon_0]$,
and the estimate
\begin{align*}
|\tilde{s}_\varepsilon'(x)-\tilde{s}_\varepsilon'(y) |
&\textstyle
\beq\Big|{\big|\frac{L}{L_\varepsilon}\g_\varepsilon'(\tilde{s}_\varepsilon(x))\big|^{-1}}-{\big|\frac{L}{L_\varepsilon}\g_\varepsilon'(\tilde{s}_\varepsilon(y))\big|^{-1}}\Big| 
\leq \big|\frac{L}{L_\varepsilon} \big|\frac{\left|\g_\varepsilon'(\tilde{s}_\varepsilon(x))-\g_\varepsilon'(\tilde{s}_\varepsilon(y))\right|}{\left|\g_\varepsilon'(\tilde{s}_\varepsilon(x))\right|\left|\g_\varepsilon'(\tilde{s}_\varepsilon(y))\right|}\\
&\textstyle
\bleq[\eqref{bounds_derivative}] 4\left|\frac{L}{L_\varepsilon} \right| \left|\g_\varepsilon'(\tilde{s}_\varepsilon(x))-\g_\varepsilon'(\tilde{s}_\varepsilon(y))\right|,
\end{align*}
we obtain for all $\varepsilon\in (0,\varepsilon_0]$ the inequality 
\begin{align}\label{J}
J_\varepsilon(x,y)\leq \textstyle
2^{(1+s)\rho}\big|\frac{L}{L_\varepsilon}\big|^\rho\big[2^\rho+4^\rho\big|\frac{L}{L_\varepsilon}\big|^\rho \big]\frac{\left|\g_\varepsilon'(\tilde{s}_\varepsilon(x))-\g_\varepsilon'(\tilde{s}_\varepsilon(y))\right|^\rho}{|\tilde{s}_\varepsilon(x)-\tilde{s}_\varepsilon(y)|_{\Sl}^{1+s\rho}} \bleq[\eqref{constante_bruch_L}]c A_\varepsilon(\psi_\varepsilon(x,y))
\end{align}
for $\textstyle A_\varepsilon(x,y)
:=\frac{\left|\g_\varepsilon'(x)
-\g_\varepsilon'(y)\right|^\rho}{|x-y|_{\Sl}^{1+s\rho}}$ and the transformation 
$\psi_\varepsilon : (\Sl)^2 \to (\Sl)^2$ sending $(x,y)$ to 
$(\tilde{s}_\varepsilon(x),\tilde{s}_\varepsilon(y))$. 
Observe, that by \eqref{bounds_derivative}, $\psi_\varepsilon$ is 
bilipschitz, since $\text{det}(D\psi_\varepsilon(x,y))=
|\tilde{s}_\varepsilon'(x)||\tilde{s}_\varepsilon'(y)|
\in [\tfrac{1}{4},4]$. This implies by \eqref{J} that
\begin{align*}
J_\varepsilon(x,y)\leq 4c A_\varepsilon(\psi_\varepsilon(x,y))\left|\text{det}(D\psi_\varepsilon(x,y))\right|\quad
\Foa \varepsilon \in (0,\varepsilon_0].
\end{align*}
Let now $E \subset \left(\Sl\right)^2$. By change of variables
\begin{align}
\label{absch_erster_Term_I}
\textstyle \int\int_E J_\varepsilon(x,y) \d x \d y&\leq \textstyle 4c \int\int_E A_\varepsilon(\psi_\varepsilon(x,y))\left|\text{det}(D\psi_\varepsilon(x,y))\right| \d x \d y \nonumber\\
&\textstyle
=4c \int\int_{\psi_\varepsilon(E)} A_\varepsilon(x,y) \d x \d y
\quad\Foa \varepsilon\in (0,\varepsilon_0].
\end{align}
It is well-known that the standard convolution $\g_\varepsilon$ 
converges in $W^{1+s, \rho}$ to $\g$; see, e.g., 
\cite[Lemma 11]{fiscella-etal_2015}, which according to
Vitali's theorem implies that
 the $A_\varepsilon(x,y)$
are 
 uniformly integrable. 
In particular, for given $\tilde{\varepsilon}>0$ there exists
$\delta_2=\delta_2(\tilde{\varepsilon})>0$ such that 
if $|\psi_\varepsilon(E)|<\delta_2$, we have
\begin{align}
\label{first_term_integrable}
\textstyle\int\int_{\psi_\varepsilon(E)} A_\varepsilon(x,y)
\d x \d y < \frac{\tilde{\varepsilon}}{c 2^{\rho+2}}\quad\Foa \varepsilon>0.
\end{align}
Since $\psi_\varepsilon$ is uniformly Lipschitz-continuous 
for  $\varepsilon\in (0,\varepsilon_0]$, 
there exists a $\delta_3>0$ such that $|E| <\delta_3$ 
implies $|\psi_\varepsilon(E) |<\delta_2$. 
Now set $\tilde{\delta}:=\min\{ \delta_1, \delta_3\}$ 
so that for any set $E\subset( \Sl)^2$ with $|E|<\tilde{\delta}$ 
we infer by means of \eqref{obvious}, \eqref{absch_erster_Term_I}, 
\eqref{second_term_integrable}, and \eqref{first_term_integrable} 
\begin{align*}
\textstyle\int\int_E I_\varepsilon(x,y)\d x \d y 
\bleq[\eqref{obvious},\eqref{absch_erster_Term_I}]& 2^{\rho-1} 
\textstyle
\Big[4c \int\int_{\psi_\varepsilon(E)} 
A_\varepsilon(x,y)
\d x \d y
+ \int\int_E \frac{\left|\gamma'(x)-\gamma'(y)\right|^\rho}{|x-y|_{\Sl}^{1+s\rho}}\d x \d y\Big]\\
\ble[\eqref{second_term_integrable},\eqref{first_term_integrable}]&
\textstyle
2^{\rho-1} \left[4c \frac{\tilde{\varepsilon}}{c 2^{\rho+2}} + \frac{\tilde{\varepsilon}}{2^\rho}\right]
\beq \tilde{\varepsilon}\quad\Foa \varepsilon\in (0,\varepsilon_0].
\end{align*}
Hence, 
$\left(I_\varepsilon\right)_{\varepsilon\in (0,\varepsilon_0]}$ 
is uniformly integrable.
\hfill $\Box$

\section{Quantitative analysis of $C^1$-curves}\label{app:B}

\begin{lemma}[Injective $C^1$-curves are bilipschitz]
\label{lem:bilipschitz}
For any curve $\g\in C^1_\textnormal{ia}(\R/L\Z,\R^3)$ there is a constant
$c_\g>0$ such that
\begin{equation}\label{eq:bilipschitz}
|t-s|_{\R/L\Z}\le c_\g|\g(t)-\g(s)|\quad\Foa t,s\in\R.
\end{equation}
\end{lemma}
\begin{proof}
Using  the  Taylor expansion
\begin{equation}\label{eq:taylor}
\textstyle
\g(s+h)-\g(s)=\int_s^{s+h}\g'(\tau)\d \tau =
\g'(s)h+\int_s^{s+h}(\g'(\tau)-\g'(s))\d \tau
\end{equation}
we choose $h_0=h_0(\g)\in (0,\frac{L}2]$ such that  
$\omega_{\g'}(h_0)\le
\frac12$
to infer for all $s\in\R$
\begin{equation}\label{eq:local-bilipschitz}
\textstyle
|\g(s+h)-\g(s)|\ge (1-\omega_{\g'}(h_0))|h|\ge\frac12 |h|=
\frac12 |h|_{\R/L\Z}\quad\Foa |h|\le h_0.
\end{equation}
On the other hand, since $\g$ is injective,
we find a constant $\delta_0=\delta_0(\g)>0$ such 
that 
\begin{equation}\label{eq:global-bilipschitz}
\textstyle
|\g(s+h)-\g(s)|\ge \delta_0\ge\frac{2\delta_0}L |h|
=\frac{2\delta_0}L |h|_{\R/L\Z}\quad\Foa s\in\R,\,
h_0\le |h|\le\frac{L}2.
\end{equation}
which implies \eqref{eq:bilipschitz} for $c_\g:=\max\{2,
\frac{L}{2\delta_0}\}.$\\
\end{proof}

\begin{lemma}
\label{absch_sk}
Let $\gamma \in C^1_\textnormal{ia}\left(\Sl,\R^3\right)$ and
$h\in (0,\frac{L}2]$ such that $\wg(h)<1$ and $h\leq \frac{L}{2}$. 
Then, for every $s \in \R$, 
\begin{align}
1-\wg(h)&\leq \textstyle\frac{\left\langle \gamma(s+h)-\gamma(s),\gamma'(s)\right\rangle}{\diffgamma{s+h}{s}} \leq 1 \label{bruch_s},\\
1-\wg(h)&\leq \textstyle\frac{\left\langle \gamma(s+h)-\gamma(s),\gamma'(s+h)\right\rangle}{\diffgamma{s+h}{s}} \leq 1.\label{bruch_s+h}
\end{align}
\begin{proof}
Since $0<h\leq \frac{L}{2}$  we have $|s+h-s|_{\Sl} = h$;
see \eqref{eq:periodic-norm}.
We use the Taylor expansion \eqref{eq:taylor} to estimate
\begin{align}
\label{absch_skalarprodukt_s}
\big\langle \gamma(s+h)-\gamma(s),\gamma'(s)\big\rangle
\geq h\Big[1-\sup_{\tau \in [s,s+h]}|\gamma'(\tau)-\gamma'(s)|\Big] 
\geq h\left[1- \wg(h)\right]
\end{align}
and analogously
\begin{align}
\label{absch_skalarprodukt_s+h}
\left\langle \gamma(s+h)-\gamma(s),\gamma'(s+h)\right\rangle
&\geq h\left[1- \wg(h)\right].
\end{align}
Using the above estimate for the inner product and
the Lipschitz estimate $|\g(s+h)-\g(s)|\le h$ we can deduce
\begin{align*}
\textstyle\frac{\left\langle \gamma(s+h)-\gamma(s),\gamma'(s)\right\rangle}{\diffgamma{s+h}{s}}\bgeq[\eqref{absch_skalarprodukt_s}] \frac{h(1-\wg(h))}{|\gamma(s+h)-\gamma(s)|}\geq \frac{h(1-\wg(h))}{h}= 1-\wg(h).
\end{align*}
Applying the Cauchy-Schwarz inequality yields the right part of inequality \eqref{bruch_s}.
Thus, statement \eqref{bruch_s} is shown. In the same manner we can conclude the statement \eqref{bruch_s+h} with the Cauchy-Schwarz inequality and estimate \eqref{absch_skalarprodukt_s+h}.
\end{proof}
\end{lemma}

\begin{lemma}
\label{c1_norm_bogenlaenge} \cite[cf. Lemma B.8]{strzelecki-vdm_2023}
Let $\gamma \in C^1\left([a,b],\R^3\right)$ satisfy $|\gamma'|\geq v_\gamma >0$ and  $\mathscr{L}(\gamma)>0$ and let $\Gamma \in C^1\left([0,\mathscr{L}(\gamma)],\R^3\right)$ be the arclength parametrization. Suppose that $\beta \in C^1\left([a,b],\R^3 \right)$ has equal length, i.e. 
$\mathscr{L}(\gamma)=\mathscr{L}(\beta)$, and satisfies 
\begin{align}
\label{distance}
\textstyle\| \gamma - \beta \|_{C^1\left([a,b],\R^3 \right)} < \varepsilon \leq \frac{v_\gamma}{2}.
\end{align}
Then $\beta$ possesses an arclength parametrization $B\in C^1\left([0,\mathscr{L}(\gamma)],\R^3 \right)$ with
\begin{align}\label{B9A}
\textstyle
\| \Gamma - B \|_{C^1\left([0,\mathscr{L}(\gamma)],\R^3 \right)} \leq  \frac{2}{v_\gamma}\omega_{\gamma'}\left(\frac{(b-a)\varepsilon}{v_\gamma} \right) + \omega_{\gamma}\left(\frac{(b-a)\varepsilon}{v_\gamma} \right) + \varepsilon \left(1+ \frac{2}{v_\gamma}\right),
\end{align}
where $\omega_{\gamma}$ denotes the modulus of continuity of $\gamma$ and $\omega_{\gamma'}$ denotes the modulus of continuity of the tangent $\gamma'$ of $\gamma$.
\begin{proof}
Without loss of generality, we can assume $\Gamma(s)=\gamma(t(s))$ for $s \in [0,\mathscr{L}(\gamma)]$ where $t : [0,\mathscr{L}(\gamma)] 
\to [a,b]$ is 
the inverse function of the arclength function
$
s(t):=\int_a^t |\gamma'(u)| \d u
$
for $t \in [a,b]$. Furthermore, the conditions $|\gamma'|\geq v_\gamma$ and \eqref{distance} imply
$
|\beta'(t)|\geq  \frac{v_\gamma}{2}>0
$
for all $t \in [a,b]$. Hence, the arclength function of $\beta$,
$
\sigma(t):=\int_a^t |\beta'(u)|\d u
$
is therefore also
invertible. Let $\tau : [0,\mathscr{L}(\gamma)] \to [a,b]$ be the inverse function of $\sigma$ and define the arclength parametrization of $\beta$ as
$
B(s):=\beta(\tau(s))
$
for $s\in [0,\mathscr{L}(\gamma)]$. Now fix an $s \in 
[0,\mathscr{L}(\gamma)]$. Then there exist unique $t,\tau \in [a,b]$ such that $s=s(t)=\sigma(\tau)$. This leads to
\begin{align}
0 &= \sigma(\tau) - s(t) 
=\textstyle
\int_a^{\tau} \left( |\beta'(u)| - |\gamma'(u)| \right) \d u - \int_{\tau}^t |\gamma'(u)|\d u. \label{equality}
\end{align}
Thus we can estimate
\begin{align}\label{absch_t_tau}
\textstyle
v_\gamma |t-\tau| \bleq \big|\int_{\tau}^t |\gamma'(u)|\d u  \big|
\beq[\eqref{equality}] \left|\int_a^{\tau} \left(|\beta'(u)|-|\gamma'(u)| \right)\d u \right| \bleq[\eqref{distance}] \varepsilon(b-a).
\end{align}
Now we can use \eqref{distance} and \eqref{absch_t_tau} to estimate the distance between $\Gamma$ and $B$ by
\begin{align}
\label{absch_c0}
\left|\Gamma(s)-B(s)\right| &\beq |\gamma(t)-\beta(\tau)| 
\bleq |\gamma(t)-\gamma(\tau)| + |\gamma(\tau)-\beta(\tau)| \nonumber\\
&\bleq[\eqref{distance}] \omega_{\gamma}(|t-\tau|) + \varepsilon 
\bleq[\eqref{absch_t_tau}] \textstyle
\omega_{\gamma}\big(\frac{(b-a)\varepsilon}{v_\gamma} \big) 
+ \varepsilon.
\end{align}
With 
$
\tau'(s)=\frac{1}{\sigma'(\tau(s))}=\frac{1}{\left|\beta'(\tau(s))\right|}
$
we obtain
$
B'(s)=\beta'(\tau(s))\tau'(s)=\frac{\beta'(\tau(s))}{\left|\beta'(\tau(s))\right|}
$
and analogously 
$
\Gamma'(s)=\frac{\gamma'(t(s))}{\left|\gamma'(t(s)) \right|}.
$
This leads to
\begin{align}
\label{absch_c1}
&\left|\Gamma'(s)-B'(s)\right| \beq
\textstyle
\big|\frac{\gamma'(t)}{|\gamma'(t)|} - \frac{\beta'(\tau)}{|\beta'(\tau)|} \big| \nonumber\\
\bleq& \textstyle
\frac{\left|\gamma'(t)-\gamma'(\tau) \right|}{|\gamma'(t)|} + |\gamma'(\tau)|\big| \frac{1}{|\gamma'(t)|} - \frac{1}{|\gamma'(\tau)|}\big| 
  + \frac{\left|\gamma'(\tau) - \beta'(\tau)\right|}{|\gamma'(\tau)|} + |\beta'(\tau)|\big| \frac{1}{|\gamma'(\tau)|} - \frac{1}{|\beta'(\tau)|} 
  \big| \nonumber \\
\bleq& \textstyle
\frac{1}{v_\gamma}\omega_{\gamma'}(|t-\tau|) + \frac{1}{v_\gamma}\omega_{\gamma'}(|t-\tau|)  + \frac{\varepsilon}{v_\gamma} + \frac{\varepsilon}{v_\gamma} 
\bleq[\eqref{absch_t_tau}] \frac{2}{v_\gamma}\big( \omega_{\gamma'}\big(\frac{(b-a)\varepsilon}{v_\gamma} \big) + \varepsilon \big).
\end{align}
With \eqref{absch_c0} and \eqref{absch_c1}, we deduce \eqref{B9A}.
\end{proof}
\end{lemma}

\section{Auxiliary results in the $\Gamma$-convergence framework}
\label{app:C}
{\it Proof of Lemma \ref{limsup_approx}.}\,
For a given sequence $(x^m)_m\subset X$ with properties (i) and (ii), let 
$\left(x^m_n\right)_{n \in \N}$ be a 
sequence satisfying (iii). We define the neighborhoods
\begin{align}
\label{def_Um}
U_m(x)\equiv U_m:=\big\{y \in X \; \colon \; d(x,y)<d\left(x,x^m\right)+
\textstyle\frac{1}{m} \big\}\quad\Fo m\in\N,
\end{align}
and the index sets
\begin{align*}
J_n(x)\equiv J_n
:=\big\{m \in \N \; \colon \; x^m_n \in U_m \text{ and } \mathcal{F}_n\left(x^m_n\right)\leq \mathcal{F}\left(x^m\right)+
\textstyle
\frac{1}{m} \big\}\quad\Fo n\in\N.
\end{align*}
For every $n \in \N$, we then set 
\begin{align*}
m_n:=\begin{cases}
\max \;J_n, &\text{if } J_n\neq \emptyset \text{ is bounded}\\
\min \; \left(J_n \cap \left\{m \in \N \; | \; m \geq n \right\}\right), &\text{if } J_n \text{ is unbounded}\\
1, &\text{if } J_n = \emptyset
\end{cases}
\end{align*}
{\bf Claim:  $\lim_{n\to \infty} m_n=\infty$.}\,
For the proof of this claim we distinguish three cases.

\textbf{Case 1:} There exists an index $N \in \N$ 
such that $J_n \neq \emptyset$ is bounded for all $n \geq N$. 
Suppose for contradiction that 
there exists a constant $c>0$ and a subsequence $k(n) \to \infty$ as 
$n \to \infty$ such that $m_{k(n)} \leq c$ for all $n \geq N$.
Thus, 
\begin{align}
\label{assumptions_contra}
\mathcal{F}_{k(n)}(x^m_{k(n)}) > \mathcal{F}(x^m)\textstyle
+ \frac{1}{m} \quad \text{or} \quad x^m_{k(n)} \notin U_m\quad\Foa m\ge c,\,
n\ge N.
\end{align}
If $x^{m}_{k(n)} \notin U_{m}$ holds for any $m \geq c$, we deduce 
from \eqref{def_Um} that $
d(x,x^{m}_{k(n)} )\geq d(x,x^{m}) + \frac{1}{m},$ and 
therefore, $
d(x^{m}_{k(n)},x^{m}) \geq d(x,x^{m}_{k(n)} ) - 
d(x,x^{m})  = \frac{1}{m}.$
This contradicts $d(x^{m}_{k(n)},x^{m}) \to 0$ as $n \to \infty$ in 
assumption (iii). Thus, we have  $x^{m}_{k(n)} \in U_{m}$ 
for all $m \geq c$ and consequently 
$\mathcal{F}_{k(n)}(x^{m}_{k(n)})
> \mathcal{F}(x^{m}) + \frac{1}{m}$ in 
\eqref{assumptions_contra} must hold. By assumption 
(iii) we have
$
\limsup_{n \to \infty} \mathcal{F}_n(x^{m}_n)\leq 
\mathcal{F}(x^{m})$, or equivalently $\inf_{n \in \N} \sup_{k \geq n}
\mathcal{F}_k(x^{m}_k) \leq \mathcal{F}(x^{m}).$
By the definition of the infimum, there exists an
index $n_0=n_0(m) \in \N$ such that
\begin{align*}
\textstyle\sup_{k \geq n_0} \mathcal{F}_k\left(x^{m}_k\right) \leq \inf_{n \in \N} \sup_{k \geq n} \mathcal{F}_k\left(x^{m}_k\right) + \frac{1}{m} \leq \mathcal{F}\left(x^{m}\right)+\frac{1}{m}
\end{align*}
and thus
$
\mathcal{F}_k(x^{m}_k) \leq \mathcal{F}(x^{m})+\frac{1}{m}$
holds for all $k \geq n_0$, which contradicts the first assumption in 
\eqref{assumptions_contra}. Hence, we conclude that the assumption was wrong and $\lim_{n \to \infty} m_n = \infty$ must hold in this case.\\
\textbf{Case 2:} $J_n=\emptyset$ for infinitely many $n\in\N$.
Then for these $n$ we have, similarly as in \eqref{assumptions_contra},
either $\mathcal{F}_n(x_n^m)>\mathcal{F}(x^m)+\frac1m$ or $x^m_n\in
U_m$ for all $m$. Exactly as 
in Case 1 both options lead to contradictive
statements, which rules out Case 2.

\textbf{Case 3:} There exists a subsequence $n_l$ such that $J_{n_l}$ is unbounded. By definition of $J_{n_l}$, we get $m_{n_l} \geq n_l$ and thus $m_{n_l} \to \infty$ as $l \to \infty$.
Every other subsequence of $m_n$ can then be handled as in case 1. 
By the subsequence principle we get 
$m_n \to \infty$ as $n$ tends to infinity, which proves the claim.

We now define $y_n :=x^{m_n}_n$ for all $n \in \N$. By the definition of $m_n$, we get $y_n \in U_{m_n}$ and thus
$
0\leq d(x,y_n)<d(x,x^{m_n})+\frac{1}{m_n}
$
for every $n \in \N$. Applying $m_n \to \infty$ as $n \to \infty$ 
and assumption (i) we then deduce $d\left(x,y_n\right) \to 0$ as $n \to \infty$. Furthermore, we estimate
\begin{align*}
\limsup_{n \to \infty} \mathcal{F}_n\left(y_n\right) 
= \limsup_{n \to \infty} \mathcal{F}_n\left(x^{m_n}_n\right) 
\leq \limsup_{n \to \infty} \big(\mathcal{F}\left(x^{m_n}\right){
\textstyle + \frac{1}{m_n}} \big)\bleq[\textnormal{(ii)}] \mathcal{F}(x).\hfill \qquad\Box 
\end{align*}

{\it Proof of Lemma \ref{konv_minimierer}.}\,
For $y\in Y$ and $y_n\to y$ with $\limsup_{n\to\infty}\mathcal{F}_n(y_n)
\le\mathcal{F}(y)$ one estimates
$$
\mathcal{F}(z)\le\liminf_{n\to\infty}\mathcal{F}_n(z_n)=
\liminf_{n\to\infty}\inf_X\mathcal{F}_n\le
\liminf_{n\to\infty}\mathcal{F}_n(y_n)\le\mathcal{F}(y).
$$
Taking the infimum over all $y\in Y$ on the right-hand side yields
the desired inequality, from which the last two equations
follow as claimed if $z\in Y$.
\hfill $\Box$

\begin{lemma}
\label{f_inf_vertauschen}
For any continuous non-decreasing map 
$g : (0,\infty) \to \R$ and any set $A \subset (0,\infty)$ one
has $\inf(g(A))=g(\inf(A))$ and $\sup(g(A))=g(\sup(A))$.
\begin{proof}
For every $a \in A$ it holds that $a \geq \inf(A)$. Since $g$ is 
non-decreasing, we have $g(a)\geq g(\inf(A))$, which yields $\inf(g(A))\geq g(\inf(A))$. For any $\varepsilon>0$ there exists  
$x \in A$ such that $x<\inf(A)+\varepsilon$. Then
\begin{align*}
g(\inf(A))\leq \inf(g(A))\leq g(x)\leq g(\inf(A)+\varepsilon)\quad\Foa
\varepsilon>0.
\end{align*}
Sending $\varepsilon$ to zero, we obtain by continuity of $g$ that 
$g(\inf(A))= \inf(g(A))$. In the same manner, we prove that $\sup(g(A))=g(\sup(A))$.
\end{proof}
\end{lemma}

\begin{lemma}
\label{gamma_konv_verkettung}
Let $(X,d)$ be a metric space, $\mathcal{F}_n,\mathcal{F} : 
X \to (0,\infty)$ and $\mathcal{F}_n 
\stackrel{\Gamma}{\underset{n \to \infty}{\longrightarrow}} \mathcal{F}$.
Furthermore, let $g : (0,\infty) \to \R$ be a continuous and monotone increasing map. Then $g \circ \mathcal{F}_n \stackrel{\Gamma}{\underset{n \to \infty}{\longrightarrow}} g \circ \mathcal{F}$.
\begin{proof}
First, we prove the liminf inequality. Let $x \in X$ 
and $x_n \to x$ for $n \to \infty$. By assumption, we have 
$\mathcal{F}(x)\leq \liminf_{n \to \infty} \mathcal{F}_n(x_n)$. This yields with the monotonicity and continuity of $g$ and Lemma \ref{f_inf_vertauschen}
\begin{align*}
g(\mathcal{F}(x))&\leq g(\liminf_{n \to \infty} 
\mathcal{F}_n(x_n))=g(\lim_{n \to \infty} \inf_{m \geq n} 
\mathcal{F}_m(x_m))=\lim_{n \to \infty}g( \inf_{m \geq n} 
\mathcal{F}_m(x_m))\\
&=\lim_{n \to \infty}\inf_{m \geq n} g( 
\mathcal{F}_m(x_m))=\liminf_{n \to \infty} g(\mathcal{F}_n(x_n)).
\end{align*}
For the limsup inequality fix $x\in X$. By assumption, there 
exists a recovery sequence $x_n \to x$ as $n \to \infty$ such that 
$\mathcal{F}(x)\geq \limsup_{n \to \infty}\mathcal{F}_n(x_n)$. 
Then again, we obtain
\begin{align*}
g(\mathcal{F}(x))&\geq g(\limsup_{n \to \infty} 
\mathcal{F}_n(x_n))=g(\lim_{n \to \infty} \sup_{m \geq n} 
\mathcal{F}_m(x_m))=\lim_{n \to \infty}g( \sup_{m \geq n} 
\mathcal{F}_m(x_m))\\
&=\lim_{n \to \infty}\sup_{m \geq n} g( 
\mathcal{F}_m(x_m))=\limsup_{n \to \infty} g(\mathcal{F}_n(x_n)).
\end{align*}
\end{proof}
\end{lemma}

 \section*{Acknowledgments}
We are indebted to Henrik Schumacher  for
 fruitful discussions on possible discretizations of
 knot energies. The first author is supported
 by {the DFG}-Gra\-du\-ier\-ten\-kolleg \emph{Energy, Entropy, and 
 Dissipative Dynamics (EDDy)},  project no. 320021702/GRK2326.
The second 
author's work was partially funded by 
DFG Grant no.\@ Mo 966/7-1 
\emph{Geometric curvature functionals: 
energy landscape and discrete methods} 
(project
no. 282535003),
and by the Excellence Initiative of the 
German federal and state governments.

\addtocontents{toc}{\SkipTocEntry}
\bibliographystyle{acm}
\bibliography{../bib-files/refs-bookproj}



\end{document}